\documentclass[11pt]{article}

\usepackage[a4paper,margin=1in]{geometry}
\usepackage[T1]{fontenc}
\usepackage[utf8]{inputenc}
\usepackage{lmodern}
\usepackage{microtype}
\usepackage{amsmath,amssymb,amsfonts,amsthm,mathtools}
\usepackage{bm}
\usepackage{authblk}
\usepackage{graphicx}
\usepackage{booktabs}
\usepackage{array}
\usepackage{multirow}
\usepackage{subcaption}
\usepackage{enumitem}
\usepackage{algorithm}
\usepackage{algpseudocode}
\usepackage{xcolor}
\usepackage{placeins}
\usepackage[backend=biber,style=numeric,sorting=none,maxbibnames=99]{biblatex}
\usepackage{hyperref}
\hypersetup{colorlinks=true, linkcolor=blue!60!black, citecolor=blue!60!black, urlcolor=blue!60!black}

\newcommand{\Hq}{\mathbb{H}}
\newcommand{\R}{\mathbb{R}}
\newcommand{\C}{\mathbb{C}}

\newcommand{\eps}{\varepsilon}
\newcommand{\pinv}{\dagger}

\DeclareMathOperator{\Sdet}{Sdet}

\DeclareMathOperator{\vol}{vol}
\DeclareMathOperator*{\argmax}{arg\,max}
\DeclareMathOperator*{\argmin}{arg\,min}

\newcommand{\maybeincludegraphics}[2][]{%
\IfFileExists{Figures/#2}{\includegraphics[#1]{Figures/#2}}{%
\IfFileExists{figs/#2}{\includegraphics[#1]{figs/#2}}{%
\fbox{\parbox{0.82\linewidth}{\centering Missing figure file: \texttt{\detokenize{#2}}.\\ Upload it in the \texttt{Figures/} folder.}}}}}

\newtheorem{theorem}{Theorem}
\newtheorem{lemma}{Lemma}
\newtheorem{proposition}{Proposition}
\newtheorem{corollary}{Corollary}
\theoremstyle{definition}
\newtheorem{definition}{Definition}
\theoremstyle{remark}
\newtheorem{remark}{Remark}

\title{Quaternion Maximum-Volume Submatrix Selection with Applications to Multichannel Imaging and Visual Data}

\author[1]{Vsevolod Kliushev\thanks{Email: \texttt{kadaverciant@gmail.com}}}
\author[2]{Junjun Pan\thanks{Email: \texttt{junjpan@hkbu.edu.hk}}}
\author[1]{Valentin Leplat\thanks{Email: \texttt{valentin.leplat@gmail.com}}}

\affil[1]{Institute of Data Science and Artificial Intelligence, Innopolis University}
\affil[2]{Department of Mathematics, Faculty of Science, Hong Kong Baptist University, Hong Kong, China}

\date{}

\begin{document}
\maketitle

\begin{abstract}
Low-rank approximation based on selected rows and columns is a useful alternative to singular value decompositions when the goal is to obtain an interpretable and compact representation of a matrix. A standard way to choose these rows and columns is the maximum-volume principle. It selects submatrices with large volume, which usually leads to stable interpolation coefficients and accurate CUR-type approximations. In this paper, we study this idea for quaternion matrices. This setting is natural for color images, three-dimensional motion data, and multi-channel signals, but it requires some care because quaternion multiplication is noncommutative. We define quaternion maximum-volume submatrix selection using quaternion singular values and the Study determinant. We then derive quaternion rank-one update formulas and use them to build two selection procedures: a greedy square-core method for row and column replacement, and a rectangular method that enlarges a selected row set until the interpolation coefficients are controlled. We prove that successful row and column swaps increase the quaternion volume of the selected square core when the exact quaternion inverse is used. We also connect the stopping criterion with quasi-dominance, prove an exact quaternion CUR identity in the full-rank case, and derive an interpolation stability bound for the associated CUR approximation. For the rectangular case, we derive an append-row pseudoinverse update and show how it leads to a natural right preconditioner for overdetermined quaternion least-squares problems. Finally, we illustrate the methods on three applications: quaternion CUR approximation of RGB images, RectMaxVol-based preconditioning for ill-conditioned quaternion least-squares systems, and row selection in quaternion motion-capture data. The experiments show that the proposed quaternion MaxVol and RectMaxVol methods provide stable and efficient selection routines.

\end{abstract}


\section{Introduction}

Low-rank approximation is a basic tool in numerical linear algebra and data analysis. Given a large matrix, the goal is to replace it by a simpler object that keeps the main information. The singular value decomposition gives the best low-rank approximation in the spectral and Frobenius norms. However, it does not preserve the rows and columns of the original matrix, and it can be expensive when the matrix is large.

An alternative is to approximate a matrix using only selected rows and columns. This leads to CUR and cross approximation methods. If $A\in\C^{m\times n}$ and if $I$ and $J$ are row and column index sets, one writes
\[
        A \approx C U R, \qquad C=A[:,J], \quad R=A[I,:],
\]
where the small core matrix $U$ is usually built from the intersection submatrix $B=A[I,J]$. This type of approximation is attractive because it is interpolatory: it uses actual rows and columns of the input data. The quality of the approximation depends strongly on the selected submatrix $B$.

The maximum-volume principle is a classical way to choose such a submatrix. In the square case, it seeks a submatrix with large determinant in modulus. In the rectangular case, the volume is usually defined as the product of singular values. Large-volume submatrices tend to be well conditioned and to produce bounded interpolation coefficients. This is the main reason why MaxVol-type methods are useful in CUR decompositions, skeleton approximations, least-squares problems, and related selection tasks \cite{goreinov1997pseudo,goreinov2001maximal,goreinov2010find,mikhalev2018rectangular,allen2024maximal}.

In this paper, we study maximum-volume submatrix selection for quaternion matrices. Quaternions are useful when each data entry has several coupled components. For example, an RGB image can be encoded as a pure quaternion matrix by assigning the red, green and blue channels to the three imaginary units. Three-dimensional motion data can also be represented by pure quaternions. In such applications, processing the data directly in the quaternion domain may preserve the coupling between channels more naturally than applying a real-valued method independently to each component.

The quaternion setting is not a direct copy of the real or complex case. The main difficulty is noncommutativity. The order of factors in inverse formulas matters, and determinant-like quantities have to be chosen carefully. Moreover, when we update a selected core by replacing one row or one column, the usual Sherman--Morrison and Woodbury formulas must be written in an order that remains valid over $\Hq$.

Our goal is not to solve the globally optimal maximum-volume problem. This problem is already difficult in the real and complex cases. Instead, we focus on practical MaxVol-type algorithms which find locally dominant or near-dominant submatrices. This is the setting in which classical MaxVol methods are used in practice.

\subsection{Contribution}

The contributions of the paper are as follows.
\begin{enumerate}[label=(\roman*)]
    \item We formulate square and rectangular maximum-volume submatrix selection for quaternion matrices using quaternion singular values and the Study determinant.

    \item We derive quaternion rank-one update formulas for row and column replacements. We also derive an append-row pseudoinverse update for the RectMaxVol procedure.

    \item We adapt Greedy MaxVol and RectMaxVol to quaternion matrices. The algorithms are written so that the multiplication order is explicit.

    \item We prove that, in exact arithmetic and for a nonsingular square core, every successful row or column swap increases the quaternion volume of the selected core.

    \item We prove a simple interpolation stability bound for the quaternion CUR approximation associated with a selected core. This shows how the quasi-dominance condition controls the loss with respect to the best approximation using the selected columns.

    \item We show why RectMaxVol can be used as a right preconditioner for overdetermined quaternion least-squares problems. In particular, we derive a simple condition-number bound from the RectMaxVol coefficient control.

    \item We evaluate the methods on RGB image approximation, quaternion least-squares preconditioning, and motion-capture row selection.
\end{enumerate}

We also implemented quaternion adaptations of GMVA and adaptive Block DEIM. These methods are natural extensions of existing real-valued adaptive selection routines. However, in our current tests they did not give a clear advantage over the simpler Greedy MaxVol and RectMaxVol methods. For this reason, we keep the main paper focused on the two basic methods. The other adaptations are briefly discussed in Appendix~\ref{app:other_algorithms}.

\subsection{Related work}

MaxVol ideas were introduced in the context of pseudo-skeleton approximations and low-rank approximation by Goreinov, Tyrtyshnikov, Zamarashkin and coauthors \cite{goreinov1997pseudo,goreinov2001maximal,goreinov2010find}. These works showed that submatrices with large volume lead to stable interpolation and useful error bounds. The rectangular MaxVol method of Mikhalev and Oseledets \cite{mikhalev2018rectangular} extended this idea to tall submatrices and made it useful for least-squares preconditioning. More recent works studied structured matrices \cite{cortinovis2020maximum}, image compression and least-squares applications \cite{allen2024maximal}, and adaptive or block variants of submatrix selection \cite{wang2013improving,gidisu2025block}.

Quaternion matrix computations have also received attention. Basic facts on quaternion matrices and the quaternion singular value decomposition can be found in \cite{zhang1997quaternions,le2004singular}. Structure-preserving QR methods were studied in \cite{jia2018new}, and large-scale QSVD algorithms were developed in \cite{jia2019lanczos}. Randomized quaternion low-rank approximation has also been considered in \cite{liu2022randomized,ahmadi2025pass}. Iterative methods for quaternion pseudoinverses were recently studied in \cite{leplat2025iterative}. To the best of our knowledge, MaxVol and RectMaxVol selection methods have not been developed systematically for quaternion matrices.

Most of the MaxVol, RectMaxVol, DEIM, and CUR theory cited above is formulated over the real or complex field. Our contribution is not to replace this classical theory, but to extend its main local mechanisms to quaternion matrices in a form that is consistent with quaternion linear algebra. This requires some care. The usual determinant is not directly available over $\mathbb H$, scalar factors do not commute, and inverse-update formulas have to be written with the correct multiplication order. We therefore use quaternion singular values, the Study determinant, and the complex embedding as proof tools, while the selection procedures themselves remain quaternion row and column selection procedures. In this sense, several results below are quaternion analogues of known real/complex MaxVol facts, but their validity over \(\mathbb H\) has to be checked rather than assumed.
\color{black}

\subsection{Organization}

The paper is organized as follows. Section~\ref{sec:preliminaries} recalls the quaternion notation, the volume definition, and the submatrix selection problem. Section~\ref{sec:algorithms} presents the quaternion Greedy MaxVol and RectMaxVol algorithms. Section~\ref{sec:theory} gives the theoretical guarantees. Section~\ref{sec:experiments} presents the numerical experiments. Section~\ref{sec:conclusion} concludes the paper. Appendix~\ref{app:other_algorithms} briefly describes the additional quaternion adaptations which were tested but are not central to the paper.

\section{Quaternion preliminaries and problem setting}
\label{sec:preliminaries}

We recall only the notation needed in the paper. A quaternion has the form
\[
    q=q_0+q_1 i+q_2 j+q_3 k,
\]
where $q_0,q_1,q_2,q_3\in\R$ and
\[
    i^2=j^2=k^2=ijk=-1.
\]
The conjugate is $q^*=q_0-q_1 i-q_2 j-q_3 k$, and the modulus is
\[
    |q| = (q^*q)^{1/2} = \left(q_0^2+q_1^2+q_2^2+q_3^2\right)^{1/2}.
\]
If $q\neq 0$, then $q^{-1}=q^*/|q|^2$. This inverse is two-sided. For matrices over $\Hq$, we denote by $A^*$ the conjugate transpose. A matrix $Q$ is unitary if $Q^*Q=I$.

Any quaternion matrix $A\in\Hq^{m\times n}$ can be written as
\[
    A=A_0+A_1 i+A_2 j+A_3 k,
\]
where $A_0,A_1,A_2,A_3$ are real matrices. It can also be written as $A=X+Yj$, where $X,Y\in\C^{m\times n}$. We use the standard complex embedding
\begin{equation}
\label{eq:complex_embedding}
    \Phi(A)=
    \begin{bmatrix}
        X & Y \\
        -\overline{Y} & \overline{X}
    \end{bmatrix}
    \in \C^{2m\times 2n}.
\end{equation}
This embedding is injective and multiplicative, that is,
\[
    \Phi(AB)=\Phi(A)\Phi(B)
\]
whenever the product is defined.

The quaternion singular value decomposition states that any $A\in\Hq^{m\times n}$ admits a factorization
\[
    A=U\Sigma V^*,
\]
where $U$ and $V$ are quaternion unitary matrices and $\Sigma$ is real diagonal, with nonnegative diagonal entries. These diagonal entries are the quaternion singular values of $A$ \cite{zhang1997quaternions,le2004singular}.

\subsection{Quaternion volume}

For a tall matrix $A\in\Hq^{m\times n}$ with $m\ge n$, we define
\begin{equation}
\label{eq:rect_volume}
    \vol(A)=\prod_{i=1}^n \sigma_i(A),
\end{equation}
where $\sigma_i(A)$ are the quaternion singular values. In the square case, this definition is consistent with determinant-based definitions. If $B\in\Hq^{k\times k}$, we use the Study determinant
\[
    \Sdet(B)=|\det(\Phi(B))|,
\]
see \cite{aslaksen1996quaternionic}. Then
\begin{equation}
\label{eq:square_volume}
    \vol(B)=\sqrt{\Sdet(B)}.
\end{equation}
Equivalently, $\vol(B)$ is the product of the quaternion singular values of $B$. For Hermitian quaternion matrices, this is also related to the Moore determinant \cite{cohen2000quaternionic}. In particular, for a scalar quaternion $q$, one has
\[
    \vol(q)=|q|.
\]

This definition is chosen so that the volume remains a real nonnegative quantity and agrees with the product of quaternion singular values. In the square case, the Study determinant provides the determinant-like object needed for volume comparisons.

\subsection{Square and rectangular selection problems}

Let $A\in\Hq^{m\times n}$. In the square case, we select row and column index sets $I$ and $J$, with $|I|=|J|=k$, and form the core
\[
    B=A[I,J]\in\Hq^{k\times k}.
\]
The goal is to find a core with large volume. In practice, we seek a locally dominant core, not a globally optimal one.

Assume that $B$ is nonsingular. We define the two coefficient matrices
\begin{equation}
\label{eq:coeff_matrices}
    C=A[:,J]B^{-1}\in\Hq^{m\times k},
    \qquad
    C'=B^{-1}A[I,:]\in\Hq^{k\times n}.
\end{equation}
A square core is locally dominant if all admissible row and column replacement coefficients have modulus at most one. In numerical algorithms, this is relaxed to a tolerance $1+\eps$.

For RectMaxVol, the column set $J$ is fixed, with $|J|=k$, and the algorithm enlarges the row set $I$ until the interpolation coefficients are controlled. If $B=A[I,J]\in\Hq^{r\times k}$ with $r\ge k$ and full column rank, the coefficient matrix is
\[
    F=A[:,J]B^\pinv.
\]
Rows with large coefficient norms are appended to $I$. This produces a tall representative submatrix.

\section{Quaternion MaxVol and RectMaxVol algorithms}
\label{sec:algorithms}

We now present the two main algorithms. The formulas are close to the real and complex ones, but the order of multiplication has to be preserved. This is the main practical point in the quaternion case.

\subsection{Rank-one inverse updates}

The quaternion Sherman--Morrison formula follows from the Woodbury identity over $\Hq$,
whose proof is given in Appendix~\ref{app:woodbury-updates}.
Let $A\in\Hq^{k\times k}$ be invertible, and let $u,v\in\Hq^k$. If $1+v^*A^{-1}u\neq 0$, then
\begin{equation}
\label{eq:quat_SM}
    (A+uv^*)^{-1}
    = A^{-1}-A^{-1}u\left(1+v^*A^{-1}u\right)^{-1}v^*A^{-1}.
\end{equation}
The scalar inverse appears in the middle of the product. This position is important because quaternions do not commute.
Thus the formulas are formally close to the real and complex Sherman--Morrison updates, but the order of the factors is part of the statement. Moving the scalar inverse to a different position would not be valid in general over \(\mathbb H\).

If the $q$th column of $A$ is replaced by a new column $b$, then
\[
    A_{\rm new}=A+(b-a_q)e_q^*,
\]
where $a_q$ is the old $q$th column. Formula~\eqref{eq:quat_SM} gives
\begin{equation}
\label{eq:column_update}
    A_{\rm new}^{-1}
    =A^{-1}-A^{-1}(b-a_q)
    \left(1+e_q^*A^{-1}(b-a_q)\right)^{-1}e_q^*A^{-1}.
\end{equation}
Similarly, if the $q$th row is replaced by a new row $b^*$, then
\begin{equation}
\label{eq:row_update}
    A_{\rm new}^{-1}
    =A^{-1}-A^{-1}e_q
    \left(1+(b^*-a_q^*)A^{-1}e_q\right)^{-1}(b^*-a_q^*)A^{-1}.
\end{equation}
These two updates are used in the square Greedy MaxVol routine. They are written separately because row and column replacements do not lead to the same
multiplication order over $\Hq$.

\subsection{Quaternion Greedy MaxVol}

The Greedy MaxVol algorithm alternates between row and column phases. In the row phase, it computes $C=A[:,J]B^{-1}$ and looks for an entry $C_{pq}$ with large modulus. If $|C_{pq}|>1+\eps$, the $q$th selected row is replaced by row $p$. The column phase is the same idea applied to $C'=B^{-1}A[I,:]$.

\begin{algorithm}[ht!]
\caption{Quaternion Greedy MaxVol}
\label{alg:q_greedy_maxvol}
\begin{algorithmic}[1]
\Require $A\in\Hq^{m\times n}$, core size $k$, tolerance $\eps$, maximum number of sweeps $s_{\max}$
\Ensure row indices $I$, column indices $J$, core $B=A[I,J]$
\State Initialize $I$ and $J$ by row/column norms or by quaternion QR pivots.
\State $B\gets A[I,J]$ and compute $B^{-1}$, for example via the complex embedding $\Phi$.
\For{$s=1,\ldots,s_{\max}$}
    \State $\mathrm{changed}\gets \mathrm{false}$
    \While{true}
        \State $C\gets A[:,J]B^{-1}$
        \State $(p,q)\gets \argmax_{i\notin I,\,1\le j\le k}|C_{ij}|$
        \If{$|C_{pq}|\le 1+\eps$}
            \State \textbf{break}
        \EndIf
        \State Replace the $q$th selected row by $p$.
        \State Update $B$ and $B^{-1}$ using the row update~\eqref{eq:row_update}.
        \State $\mathrm{changed}\gets \mathrm{true}$
    \EndWhile
    \While{true}
        \State $C'\gets B^{-1}A[I,:]$
        \State $(q,r)\gets \argmax_{1\le i\le k,\,j\notin J}|C'_{ij}|$
        \If{$|C'_{qr}|\le 1+\eps$}
            \State \textbf{break}
        \EndIf
        \State Replace the $q$th selected column by $r$.
        \State Update $B$ and $B^{-1}$ using the column update~\eqref{eq:column_update}.
        \State $\mathrm{changed}\gets \mathrm{true}$
    \EndWhile
    \If{$\mathrm{changed}=\mathrm{false}$}
        \State \textbf{break}
    \EndIf
\EndFor
\State \Return $I,J,B$
\end{algorithmic}
\end{algorithm}

In exact arithmetic, Algorithm~\ref{alg:q_greedy_maxvol} is naturally written with the true inverse of $B$. In computations, one may recompute the inverse through the complex embedding, or use a quaternion pseudoinverse when the core is close to singular. The monotonicity result of Section~\ref{sec:theory} applies to the exact inverse case. With approximate inverses, it should be interpreted as the principle guiding the algorithm.

\subsection{Quaternion RectMaxVol}

RectMaxVol starts from a square core and then adds rows. The goal is to build a tall submatrix whose interpolation coefficients are bounded. Let $B\in\Hq^{r\times k}$ have full column rank and suppose we append a row $a^*\in\Hq^{1\times k}$. Then
\[
    B_{\rm new}=\begin{bmatrix} B\\ a^*\end{bmatrix}.
\]
Since $B$ has full column rank,
\[
    B^\pinv=(B^*B)^{-1}B^*.
\]
Let
\begin{equation}
\label{eq:rect_update_quantities}
    G=B^*B,
    \qquad c=G^{-1}a,
    \qquad d=1+a^*c,
    \qquad s=Bc.
\end{equation}
Because $G$ is Hermitian positive definite, $d\in\R_{>0}$. Hence $d$ commutes with all quaternion quantities. The append-row update is
\begin{equation}
\label{eq:append_row_pinv_update}
    B_{\rm new}^\pinv
    =\begin{bmatrix}
        B^\pinv-cd^{-1}s^* & cd^{-1}
      \end{bmatrix}.
\end{equation}
This update is cheap once $G^{-1}$ or an equivalent representation is available. Its derivation
is given in Proposition~\ref{prop:append_row_update}; it relies on the quaternion Woodbury
identity, recalled in Appendix~\ref{app:woodbury-updates}.

\begin{algorithm}[ht!]
\caption{Quaternion RectMaxVol}
\label{alg:q_rect_maxvol}
\begin{algorithmic}[1]
\Require $A\in\Hq^{m\times n}$, initial rows $I$, fixed columns $J$, threshold $\tau$, maximum number of rows $r_{\max}$
\Ensure enlarged row set $I$ and rectangular core $B=A[I,J]$
\State $B\gets A[I,J]$ and compute $B^\pinv$.
\While{$|I|<r_{\max}$}
    \State $F\gets A[:,J]B^\pinv$
    \State $\ell_i\gets \|F[i,:]\|_2$ for all $i\notin I$
    \State $p\gets \argmax_{i\notin I}\ell_i$
    \If{$\ell_p\le \tau$}
        \State \textbf{break}
    \EndIf
    \State Append $p$ to $I$ and append $A[p,J]$ to $B$.
    \State Update $B^\pinv$ using~\eqref{eq:append_row_pinv_update}.
\EndWhile
\State \Return $I,B$
\end{algorithmic}
\end{algorithm}

RectMaxVol will also be used below to build preconditioners for overdetermined quaternion least-squares problems. If $A\in\Hq^{m\times n}$ with $m\gg n$, the square variant selects $n$ rows and uses $T=B^{-1}$ as a right preconditioner. The rectangular variant selects $r\ge n$ rows, forms $S=A[I,:]$, computes a quaternion QR factorization $S=QR$, and uses $T=R^{-1}$. Then the selected rows of the preconditioned matrix satisfy
\[
    A[I,:]R^{-1}=Q,
\]
so they have orthonormal columns. This is the stabilization mechanism used in the experiments,
and it is made precise in Proposition~\ref{prop:rectmaxvol-preconditioning-bound}.

\section{Theoretical guarantees}
\label{sec:theory}



This section gives the main theoretical properties used in the paper. The logic is the following.
We first recall how the complex embedding allows us to define a two-sided inverse and a volume
for quaternion matrices. We then prove that quaternion volume is multiplicative and that the
volume of a rank-one update of the identity is controlled by a scalar quaternion. These facts imply
that every successful row or column swap in the square MaxVol procedure increases the current
core volume.

After that, we connect the stopping criterion of the algorithm with the usual notion of
dominance. This gives a simple interpretation of what the algorithm returns: not a globally
maximal-volume submatrix, but a locally dominant or quasi-dominant core. 
We also record the exact quaternion CUR identity in the full-rank case, and prove a simple interpolation stability bound for the CUR approximation associated with a selected core.

Finally, we discuss the rectangular case. We prove the append-row pseudoinverse update used by
RectMaxVol, and then show why RectMaxVol gives a natural right preconditioner for quaternion
least-squares problems. The last subsection explains how the exact theory should be interpreted
when the inverse is computed approximately.

The results in this section should therefore be read as quaternion counterparts of the local MaxVol and RectMaxVol mechanisms used in the real and complex literature. Some final identities have the same form as in the commutative setting. This similarity is useful, but it is not automatic. In the quaternion case, the proof has to account for noncommutativity, for the distinction between left and right multiplication, and for the absence of an ordinary determinant with all the usual properties. The complex embedding, the Study determinant, quaternion singular values, and Hermitian positivity are the tools that allow the classical MaxVol intuition to survive in the quaternion setting.

\subsection{Quaternion inverse through the complex embedding}

\begin{lemma}
\label{lem:inverse_embedding}
Let $B\in\Hq^{k\times k}$. If $\Phi(B)$ is invertible, then there exists a unique matrix $B^{-1}\in\Hq^{k\times k}$ such that
\[
    \Phi(B^{-1})=\Phi(B)^{-1}.
\]
Moreover, $B^{-1}$ is a two-sided inverse: $BB^{-1}=B^{-1}B=I$.
\end{lemma}


\begin{proof}
The complex embedding is injective and multiplicative. We first check that
\(\Phi(B)^{-1}\) belongs to the range of \(\Phi\). A matrix in the range of
\(\Phi\) has the block form
\[
    M=
    \begin{bmatrix}
        X & Y\\
        -\overline Y & \overline X
    \end{bmatrix}.
\]
Equivalently, it is characterized by the symmetry associated with the standard
complex representation of quaternion matrices. This symmetry is preserved under
matrix inversion. Indeed, if \(M\) satisfies it and is nonsingular, then the
identity \(MM^{-1}=I\) implies that \(M^{-1}\) satisfies the same block relation.
Hence there exist \(P,Q\in\mathbb C^{k\times k}\) such that
\[
    \Phi(B)^{-1}
    =
    \begin{bmatrix}
        P & Q\\
        -\overline Q & \overline P
    \end{bmatrix}.
\]
Therefore \(\Phi(B)^{-1}\) is the embedding of a unique quaternion matrix, which
we denote by \(B^{-1}\).

Using multiplicativity of \(\Phi\), we get
\[
    \Phi(BB^{-1})=\Phi(B)\Phi(B^{-1})=\Phi(B)\Phi(B)^{-1}=I.
\]
By injectivity, \(BB^{-1}=I\). Similarly,
\[
    \Phi(B^{-1}B)=\Phi(B)^{-1}\Phi(B)=I,
\]
and therefore \(B^{-1}B=I\). Thus \(B^{-1}\) is a two-sided inverse.
\end{proof}

\subsection{Volume identities}

\begin{lemma}[Multiplicativity of the quaternion volume]
\label{lem:volume_multiplicative}
Let $X,Y\in\Hq^{k\times k}$. With the definition $\vol(B)=\sqrt{\Sdet(B)}$, one has
\[
    \vol(XY)=\vol(X)\vol(Y).
\]
\end{lemma}

\begin{proof}
Since $\Phi$ is multiplicative,
\[
    \Phi(XY)=\Phi(X)\Phi(Y).
\]
Therefore
\[
    \Sdet(XY)=|\det(\Phi(XY))|
    =|\det(\Phi(X))|\,|\det(\Phi(Y))|
    =\Sdet(X)\Sdet(Y).
\]
Taking square roots gives the result.
\end{proof}

\begin{lemma}[Rank-one volume factor]
\label{lem:rank_one_factor}
Let $u\in\Hq^{k\times 1}$ and $v^*\in\Hq^{1\times k}$. Define
\[
    M=I+uv^*, \qquad d=1+v^*u\in\Hq.
\]
Then
\[
    \vol(M)=|d|.
\]
In particular, $M$ is invertible if and only if $d\neq 0$.
\end{lemma}

\begin{proof}
Using the complex embedding,
\[
    \Phi(M)=I_{2k}+\Phi(u)\Phi(v^*).
\]
The matrices $\Phi(u)$ and $\Phi(v^*)$ have sizes $2k\times 2$ and $2\times 2k$, respectively. By the Weinstein--Aronszajn identity,
\[
    \det(I_{2k}+\Phi(u)\Phi(v^*))
    =\det(I_2+\Phi(v^*)\Phi(u)).
\]
Since $\Phi$ preserves multiplication,
\[
    I_2+\Phi(v^*)\Phi(u)=\Phi(1+v^*u)=\Phi(d).
\]
For a scalar quaternion $d$, one has $|\det(\Phi(d))|=|d|^2$. Therefore
\[
    \Sdet(M)=|d|^2,
\]
and the result follows by taking the square root.
\end{proof}

\subsection{Volume increase for row and column swaps}

Let $A\in\Hq^{m\times n}$ and let $B=A[I,J]\in\Hq^{k\times k}$ be nonsingular. Recall the coefficient matrices
\[
    C=A[:,J]B^{-1}, \qquad C'=B^{-1}A[I,:].
\]

\begin{lemma}[Row swap]
\label{lem:row_swap}
Fix $J$ and replace the $q$th selected row of $B$ by a row indexed by $p\notin I$. Let $B_{\rm new}$ be the updated core. Then
\[
    \vol(B_{\rm new})=|C_{pq}|\,\vol(B).
\]
Hence, if $|C_{pq}|>1$, the row swap strictly increases the volume.
\end{lemma}

\begin{proof}
Let $b=A[p,J]$ be the new row restricted to the selected columns, and let $B_{q:}$ be the old $q$th row. Set $\delta=b-B_{q:}$. Then
\[
    B_{\rm new}=B+e_q\delta = B(I+B^{-1}e_q\delta).
\]
By Lemma~\ref{lem:volume_multiplicative},
\[
    \vol(B_{\rm new})=\vol(B)\vol(I+B^{-1}e_q\delta).
\]
Using Lemma~\ref{lem:rank_one_factor} with $u=B^{-1}e_q$ and $v^*=\delta$, we get
\[
    \vol(I+B^{-1}e_q\delta)=|1+\delta B^{-1}e_q|.
\]
Now
\[
    1+\delta B^{-1}e_q
    =1+(b-B_{q:})B^{-1}e_q
    =bB^{-1}e_q.
\]
The scalar $bB^{-1}e_q$ is precisely the $(p,q)$ entry of $C=A[:,J]B^{-1}$. This proves the claim.
\end{proof}

\begin{lemma}[Column swap]
\label{lem:column_swap}
Fix $I$ and replace the $q$th selected column of $B$ by a column indexed by $r\notin J$. Let $B_{\rm new}$ be the updated core. Then
\[
    \vol(B_{\rm new})=|C'_{qr}|\,\vol(B).
\]
Hence, if $|C'_{qr}|>1$, the column swap strictly increases the volume.
\end{lemma}

\begin{proof}
Let $a=A[I,r]$ be the new column restricted to the selected rows, and let $B_{:q}$ be the old $q$th column. Set $\Delta=a-B_{:q}$. Then
\[
    B_{\rm new}=B+\Delta e_q^*=(I+\Delta e_q^*B^{-1})B.
\]
By Lemma~\ref{lem:volume_multiplicative},
\[
    \vol(B_{\rm new})=\vol(I+\Delta e_q^*B^{-1})\vol(B).
\]
Using Lemma~\ref{lem:rank_one_factor} with $u=\Delta$ and $v^*=e_q^*B^{-1}$ gives
\[
    \vol(I+\Delta e_q^*B^{-1})=|1+e_q^*B^{-1}\Delta|.
\]
Since
\[
    1+e_q^*B^{-1}(a-B_{:q})=e_q^*B^{-1}a,
\]
and $e_q^*B^{-1}a$ is the $(q,r)$ entry of $C'=B^{-1}A[I,:]$, the result follows.
\end{proof}

\begin{theorem}[Monotone growth under successful quaternion MaxVol swaps]
\label{thm:monotone_swaps}
Let $A\in\Hq^{m\times n}$ and let $B=A[I,J]$ be a nonsingular $k\times k$ core. Consider a two-sided MaxVol procedure that accepts a row swap only if $|C_{pq}|>1$ and a column swap only if $|C'_{qr}|>1$. Then every accepted swap strictly increases $\vol(B)$. More precisely,
\[
\vol(B_{\rm new})=
\begin{cases}
|C_{pq}|\,\vol(B), & \text{for an accepted row swap},\\[1mm]
|C'_{qr}|\,\vol(B), & \text{for an accepted column swap}.
\end{cases}
\]
\end{theorem}

\begin{proof}
This is a direct consequence of Lemmas~\ref{lem:row_swap} and~\ref{lem:column_swap}.
\end{proof}

This theorem is the quaternion analogue of the classical MaxVol swap property. The proof shows that, once volume is defined through the Study determinant and the coefficient matrices are formed with the correct right/left order, the volume multiplier is again the modulus of the corresponding quaternion interpolation coefficient.

\begin{corollary}[No cycling in exact arithmetic]
Assume that the algorithm uses exact inverses and accepts only swaps satisfying $|C_{pq}|>1$ or $|C'_{qr}|>1$. Then the algorithm cannot revisit the same core after a successful swap.
\end{corollary}

\begin{proof}
Every successful swap strictly increases the volume. Returning to a previous core would return to its previous volume, which is impossible.
\end{proof}


\subsection{Dominance and exact quaternion CUR recovery}
\label{subsec:qcur-dominance}

The previous theorem shows that a successful row or column replacement increases the volume
of the current square core. We now connect this local volume increase with the interpolation
viewpoint behind CUR approximations. This also explains what the stopping criterion of the
algorithm means.

Let \(A\in\mathbb H^{m\times n}\), and let \(I\subset\{1,\ldots,m\}\) and
\(J\subset\{1,\ldots,n\}\) be two index sets of cardinality \(k\). We denote
\[
        B=A[I,J],\qquad C=A[:,J],\qquad R=A[I,:].
\]
When \(B\) is nonsingular, the associated skeleton approximation is
\[
        A \approx C B^{-1} R.
\]
The following proposition is the quaternion analogue of the classical exact CUR identity.

\begin{proposition}[Exact quaternion CUR recovery]
\label{prop:exact-qcur}
Let \(A\in\mathbb H^{m\times n}\) have rank \(k\), and assume that
\(B=A[I,J]\in\mathbb H^{k\times k}\) is nonsingular. Then
\[
        A = A[:,J] B^{-1} A[I,:].
\]
\end{proposition}

\begin{proof}
Let \(C=A[:,J]\). Since \(B=C[I,:]\) is nonsingular, the selected columns of \(C\)
are right-linearly independent. Since \(\operatorname{rank}(A)=k\), every column
of \(A\) belongs to the right span of the columns of \(C\). Hence there exists a
matrix \(X\in\mathbb H^{k\times n}\) such that
\[
        A=CX.
\]
Restricting this identity to the rows indexed by \(I\) gives
\[
        A[I,:]=C[I,:]X=BX.
\]
Multiplying from the left by \(B^{-1}\), we obtain
\[
        X=B^{-1}A[I,:].
\]
Substituting this expression in \(A=CX\) gives
\[
        A=A[:,J]B^{-1}A[I,:].
\]
This proves the identity.
\end{proof}

For a general matrix, the same formula gives an interpolation-based low-rank approximation.
Its stability depends on the size of the interpolation coefficients. This motivates the following
definition.

\begin{definition}[Quaternion quasi-dominant core]
\label{def:q-dominant}
Let \(B=A[I,J]\in\mathbb H^{k\times k}\) be nonsingular. Define
\[
        C_B=A[:,J]B^{-1}\in\mathbb H^{m\times k},
        \qquad
        R_B=B^{-1}A[I,:]\in\mathbb H^{k\times n}.
\]
We say that \(B\) is \(\varepsilon\)-quasi-dominant if
\[
        \max_{i,j}|(C_B)_{ij}|\leq 1+\varepsilon,
        \qquad
        \max_{i,j}|(R_B)_{ij}|\leq 1+\varepsilon.
\]
For \(\varepsilon=0\), we say that \(B\) is dominant.
\end{definition}

This definition is the direct quaternion version of the classical MaxVol dominance condition.
It is well-defined because the modulus of a quaternion coefficient is a nonnegative real number.

\begin{proposition}[Termination implies quasi-dominance]
\label{prop:termination-dominance}
Assume that the two-sided quaternion MaxVol algorithm uses exact inverses and terminates when
no admissible row or column replacement coefficient has modulus larger than \(1+\varepsilon\).
Then the final core \(B=A[I,J]\) is \(\varepsilon\)-quasi-dominant.
\end{proposition}

\begin{proof}
At termination of the row phase, no admissible entry of
\[
        C_B=A[:,J]B^{-1}
\]
has modulus larger than \(1+\varepsilon\). The rows indexed by \(I\) satisfy
\[
        C_B[I,:]=A[I,J]B^{-1}=BB^{-1}=I_k.
\]
Thus their entries have modulus at most one. Hence
\[
        \max_{i,j}|(C_B)_{ij}|\leq 1+\varepsilon.
\]

The same argument applies to the column phase. At termination, no admissible entry of
\[
        R_B=B^{-1}A[I,:]
\]
has modulus larger than \(1+\varepsilon\). The columns indexed by \(J\) satisfy
\[
        R_B[:,J]=B^{-1}A[I,J]=B^{-1}B=I_k.
\]
Thus
\[
        \max_{i,j}|(R_B)_{ij}|\leq 1+\varepsilon.
\]
The two inequalities are exactly the definition of an \(\varepsilon\)-quasi-dominant core.
\end{proof}

This proposition is local. It does not imply that the selected core has globally maximal volume.
It says that the algorithm has reached a state in which all one-row and one-column replacements
are controlled. This is the property that is useful for interpolation and numerical stability.

\subsection{Interpolation stability of the quaternion CUR approximation}
\label{subsec:cur-stability}

The quasi-dominance condition has a direct interpretation in terms of interpolation stability.
We now make this point explicit. This result is not a best rank-\(k\) error bound. It is a
conditional bound: once the columns \(C=A[:,J]\) have been selected, the CUR approximation
is controlled by the best approximation of \(A\) using these columns and by the size of the
interpolation coefficients.

This is an interpolation stability statement for the selected columns and rows. It should not be interpreted as a global relative-error CUR guarantee with respect to the best rank-\(k\) approximation. Such guarantees require additional assumptions or different sampling mechanisms.

Let \(S_I\in\mathbb R^{k\times m}\) be the row-selection matrix such that
\[
        S_I A = A[I,:].
\]
Let
\[
        C=A[:,J],\qquad B=A[I,J]=S_I C,\qquad R=A[I,:]=S_I A.
\]
Assume that \(B\) is nonsingular and define
\[
        P = C B^{-1} S_I .
\]
Then the CUR approximation can be written as
\[
        C B^{-1} R = C B^{-1}S_I A = PA.
\]
Moreover,
\[
        PC = C B^{-1}S_I C = C B^{-1}B = C,
\]
so \(P\) interpolates the selected columns.

\begin{proposition}[Interpolation stability]
\label{prop:qcur-interpolation-stability}
Let \(A\in\mathbb H^{m\times n}\), let \(C=A[:,J]\in\mathbb H^{m\times k}\), and
let \(B=A[I,J]\in\mathbb H^{k\times k}\) be nonsingular. Define
\[
        P = C B^{-1} S_I,
        \qquad
        A_{\rm CUR}= C B^{-1} A[I,:]=PA.
\]
Then
\[
        \|A-A_{\rm CUR}\|_F
        \leq
        (1+\|P\|_2)
        \min_{X\in\mathbb H^{k\times n}}
        \|A-CX\|_F .
\]
\end{proposition}

\begin{proof}
Let \(X\in\mathbb H^{k\times n}\) be arbitrary. Since \(PC=C\), we have
\[
        (I-P)CX=0.
\]
Therefore
\[
        A-A_{\rm CUR}
        =
        A-PA
        =
        (I-P)A
        =
        (I-P)(A-CX).
\]
Using the submultiplicativity of the spectral norm with the Frobenius norm gives
\[
        \|A-A_{\rm CUR}\|_F
        \leq
        \|I-P\|_2\,\|A-CX\|_F.
\]
Finally,
\[
        \|I-P\|_2\leq 1+\|P\|_2.
\]
Since the inequality holds for every \(X\), taking the minimum over
\(X\in\mathbb H^{k\times n}\) proves the result.
\end{proof}

The proposition shows that the CUR error is small when two things happen: the selected columns
span the matrix well, and the interpolation operator \(P\) is not too large. The second point is
precisely where MaxVol enters.

\begin{corollary}[Stability under quasi-dominance]
\label{cor:qcur-quasidominant-stability}
Assume that \(B=A[I,J]\) is \(\varepsilon\)-quasi-dominant. Then
\[
        \|A-A_{\rm CUR}\|_F
        \leq
        \left(1+(1+\varepsilon)\sqrt{mk}\right)
        \min_{X\in\mathbb H^{k\times n}}
        \|A-CX\|_F .
\]
\end{corollary}

\begin{proof}
By definition,
\[
        P=C B^{-1}S_I.
\]
Since \(S_I\) is a row-selection matrix, \(\|S_I\|_2=1\). Hence
\[
        \|P\|_2
        \leq
        \|CB^{-1}\|_2.
\]
If \(B\) is \(\varepsilon\)-quasi-dominant, then all entries of \(CB^{-1}\) have modulus at most
\(1+\varepsilon\). Therefore
\[
        \|CB^{-1}\|_2
        \leq
        \|CB^{-1}\|_F
        \leq
        (1+\varepsilon)\sqrt{mk}.
\]
Substituting this estimate in Proposition~\ref{prop:qcur-interpolation-stability} gives the
claim.
\end{proof}

This corollary is intentionally simple. The constant is not sharp, but the message is useful:
quasi-dominance controls the stability of the interpolation operator. To obtain a bound directly
in terms of the best rank-\(k\) approximation error, one would also need to control how well the
selected columns \(C=A[:,J]\) approximate the column space of \(A\). This is a stronger classical
MaxVol-type question and is left for future work.

\color{black}

\subsection{Append-row pseudoinverse update for RectMaxVol}

We now prove the append-row pseudoinverse update used by RectMaxVol.

\begin{proposition}[Quaternion append-row pseudoinverse update]
\label{prop:append_row_update}
Let $B\in\Hq^{r\times k}$ have full column rank and let
\[
    B_{\rm new}=\begin{bmatrix}B\\a^*\end{bmatrix},
    \qquad a\in\Hq^{k}.
\]
Let $G=B^*B$, $c=G^{-1}a$, $d=1+a^*c$, and $s=Bc$. Then $d\in\R_{>0}$ and
\[
    B_{\rm new}^\pinv
    =\begin{bmatrix}
        B^\pinv-cd^{-1}s^* & cd^{-1}
      \end{bmatrix}.
\]
\end{proposition}



\begin{proof}
Since \(B\) has full column rank, \(G=B^*B\) is Hermitian positive definite.
Thus \(G^{-1}\) is Hermitian positive definite, and
\[
    a^*G^{-1}a\in\mathbb R_{\ge 0}.
\]
Hence
\[
    d=1+a^*G^{-1}a=1+a^*c\in\mathbb R_{>0}.
\]
In particular, \(d\) commutes with all quaternionic quantities.

We have
\[
    B_{\rm new}^*B_{\rm new}=B^*B+aa^*=G+aa^*.
\]
By the quaternion Woodbury identity,
\[
    (G+aa^*)^{-1}
    =
    G^{-1}-G^{-1}a(1+a^*G^{-1}a)^{-1}a^*G^{-1}.
\]
Using \(c=G^{-1}a\) and \(d=1+a^*G^{-1}a\), this becomes
\[
    (G+aa^*)^{-1}
    =
    G^{-1}-cd^{-1}a^*G^{-1}.
\]
Now
\[
    B_{\rm new}^\dagger
    =
    (B_{\rm new}^*B_{\rm new})^{-1}B_{\rm new}^*
    =
    (G+aa^*)^{-1}
    \begin{bmatrix}
        B^* & a
    \end{bmatrix}.
\]
We compute the two blocks separately. For the first block,
\[
\begin{aligned}
    (G+aa^*)^{-1}B^*
    &=
    G^{-1}B^* - cd^{-1}a^*G^{-1}B^*  \\
    &=
    B^\dagger - cd^{-1}(B c)^*        \\
    &=
    B^\dagger - cd^{-1}s^* .
\end{aligned}
\]
For the second block,
\[
\begin{aligned}
    (G+aa^*)^{-1}a
    &=
    G^{-1}a - cd^{-1}a^*G^{-1}a  \\
    &=
    c - cd^{-1}(a^*c)             \\
    &=
    c\left(1-d^{-1}(d-1)\right)   \\
    &=
    cd^{-1}.
\end{aligned}
\]
Therefore
\[
    B_{\rm new}^\dagger
    =
    \begin{bmatrix}
        B^\dagger-cd^{-1}s^* & cd^{-1}
    \end{bmatrix}.
\]
This proves the formula.
\end{proof}

\subsection{RectMaxVol and right preconditioning}
\label{subsec:rectmaxvol-preconditioning-theory}

We now explain why RectMaxVol gives a natural right preconditioner for overdetermined
quaternion least-squares problems. This is the mechanism used in the numerical experiments.
The argument parallels the real/complex RectMaxVol preconditioning mechanism, but it is stated for right quaternion least-squares preconditioning. The QR factorization, pseudoinverse, and row-norm identities are used in the quaternion sense.

Let
\[
        A\in\mathbb H^{m\times n},\qquad m\geq n,
\]
be full column rank. We consider
\[
        \min_{x\in\mathbb H^n}\|Ax-b\|_2.
\]
Assume that RectMaxVol selects a row set \(I\) and define
\[
        S=A[I,:]\in\mathbb H^{r\times n},\qquad r\geq n.
\]
We assume that \(S\) has full column rank. Let
\[
        S=QR,\qquad Q^*Q=I_n,
\]
be a reduced quaternion QR factorization of \(S\). We use the right preconditioner
\[
        T=R^{-1}.
\]
The preconditioned least-squares problem is
\[
        \min_{y\in\mathbb H^n}\|ATy-b\|_2,
        \qquad x=Ty.
\]
The selected rows of the preconditioned matrix are
\[
        A[I,:]T = SR^{-1}=Q,
\]
so they have orthonormal columns. Thus the preconditioned matrix contains a normalized
representative subsystem.

\begin{proposition}[Conditioning bound for RectMaxVol preconditioning]
\label{prop:rectmaxvol-preconditioning-bound}
Let \(A\in\mathbb H^{m\times n}\) be full column rank. Let
\(S=A[I,:]\in\mathbb H^{r\times n}\) be full column rank, and let
\[
        S=QR,\qquad Q^*Q=I_n,
\]
be a reduced quaternion QR factorization. Set \(T=R^{-1}\).

Assume that RectMaxVol returns \(S\) with the coefficient bound
\[
        \|A[i,:]S^\dagger\|_2\leq \tau
        \qquad \text{for all } i\notin I.
\]
Then
\[
        \sigma_{\min}(AT)\geq 1,
\]
and
\[
        \sigma_{\max}(AT)
        \leq
        \sqrt{\,n+(m-r)\tau^2\,}.
\]
In particular, if \(\tau\geq 1\), then
\[
        \kappa_2(AT)
        \leq
        \sqrt{\,n+(m-r)\tau^2\,}
        \leq
        \sqrt m\,\tau.
\]
\end{proposition}

\begin{proof}
Since \(S=QR\) and \(Q^*Q=I_n\), the Moore--Penrose pseudoinverse of \(S\) is
\[
        S^\dagger=R^{-1}Q^*.
\]
Indeed,
\[
        S^\dagger=(S^*S)^{-1}S^*
        =(R^*Q^*QR)^{-1}R^*Q^*
        =(R^*R)^{-1}R^*Q^*
        =R^{-1}Q^*.
\]

Let \(a_i=A[i,:]\) be a row of \(A\). Then
\[
        a_iS^\dagger=a_iR^{-1}Q^*.
\]
Since \(Q^*Q=I_n\), right multiplication by \(Q^*\) preserves the Euclidean norm of
row vectors in \(\mathbb H^{1\times n}\). Hence
\[
        \|a_iS^\dagger\|_2
        =
        \|a_iR^{-1}Q^*\|_2
        =
        \|a_iR^{-1}\|_2
        =
        \|a_iT\|_2.
\]
The RectMaxVol coefficient bound therefore implies
\[
        \|A[i,:]T\|_2\leq \tau
        \qquad \text{for all } i\notin I.
\]

Now reorder the rows of \(AT\) so that the selected rows come first. Then
\[
        AT=
        \begin{bmatrix}
        A[I,:]T\\
        A[I^c,:]T
        \end{bmatrix}
        =
        \begin{bmatrix}
        Q\\
        E
        \end{bmatrix},
\]
where \(E=A[I^c,:]T\). Therefore
\[
        (AT)^*(AT)=Q^*Q+E^*E=I_n+E^*E.
\]
Since \(E^*E\) is Hermitian positive semidefinite, all eigenvalues of
\((AT)^*(AT)\) are not less than one. Hence
\[
        \sigma_{\min}(AT)\geq 1.
\]

For the upper bound, we use the Frobenius norm:
\[
        \|AT\|_F^2
        =
        \|Q\|_F^2+\|E\|_F^2.
\]
Since \(Q^*Q=I_n\), we have \(\|Q\|_F^2=n\). Moreover, every non-selected row of
\(E\) has Euclidean norm at most \(\tau\). Hence
\[
        \|E\|_F^2\leq (m-r)\tau^2.
\]
Thus
\[
        \sigma_{\max}(AT)
        \leq
        \|AT\|_F
        \leq
        \sqrt{\,n+(m-r)\tau^2\,}.
\]
Combining this bound with \(\sigma_{\min}(AT)\geq 1\) gives the condition-number estimate.
If \(\tau\geq 1\), then
\[
        n+(m-r)\tau^2 \leq m\tau^2,
\]
and therefore
\[
        \kappa_2(AT)\leq \sqrt m\,\tau.
\]
\end{proof}

The bound is simple and not meant to be sharp. Its role is to explain the stabilization mechanism.
RectMaxVol selects a representative set of rows. The QR factorization normalizes this selected
subsystem. The RectMaxVol stopping criterion controls the remaining rows of the preconditioned
matrix.

\subsection{On approximate inverses in the implementation}
\label{subsec:approx-inverses}

The monotonicity and dominance results above are stated for exact inverses. In practice, one may
use an inverse computed through the complex embedding, an iterative inverse, or a quaternion
pseudoinverse. It is therefore useful to record a simple perturbation observation.

Let \(B\in\mathbb H^{k\times k}\) be nonsingular and let \(U\) be an approximation of \(B^{-1}\).
Define the exact and computed row coefficient matrices by
\[
        C=A[:,J]B^{-1},
        \qquad
        \widetilde C=A[:,J]U.
\]
Similarly, define
\[
        R=B^{-1}A[I,:],
        \qquad
        \widetilde R=UA[I,:].
\]

\begin{proposition}[Coefficient perturbation under an approximate inverse]
\label{prop:approx-inverse-coefficients}
Assume that
\[
        \|U-B^{-1}\|_2\leq \eta.
\]
Then
\[
        \|\widetilde C-C\|_F
        \leq
        \|A[:,J]\|_F\,\eta,
        \qquad
        \|\widetilde R-R\|_F
        \leq
        \eta\,\|A[I,:]\|_F.
\]
Moreover, for every entry,
\[
        |\widetilde C_{ij}-C_{ij}|
        \leq
        \|A[i,J]\|_2\,\eta,
\]
and
\[
        |\widetilde R_{ij}-R_{ij}|
        \leq
        \eta\,\|A[I,j]\|_2.
\]
\end{proposition}

\begin{proof}
Let
\[
        E=U-B^{-1}.
\]
Then
\[
        \widetilde C-C=A[:,J]E.
\]
Hence
\[
        \|\widetilde C-C\|_F
        =
        \|A[:,J]E\|_F
        \leq
        \|A[:,J]\|_F\|E\|_2
        \leq
        \|A[:,J]\|_F\eta.
\]
Similarly,
\[
        \widetilde R-R=EA[I,:],
\]
and therefore
\[
        \|\widetilde R-R\|_F
        \leq
        \|E\|_2\|A[I,:]\|_F
        \leq
        \eta\|A[I,:]\|_F.
\]

For the entrywise bounds, note that
\[
        \widetilde C_{ij}-C_{ij}
        =
        A[i,J]Ee_j.
\]
Thus, by Cauchy--Schwarz,
\[
        |\widetilde C_{ij}-C_{ij}|
        \leq
        \|A[i,J]\|_2 \|Ee_j\|_2
        \leq
        \|A[i,J]\|_2 \|E\|_2
        \leq
        \|A[i,J]\|_2\eta.
\]
The proof for \(\widetilde R-R\) is the same. We have
\[
        \widetilde R_{ij}-R_{ij}
        =
        e_i^*EA[I,j],
\]
and hence
\[
        |\widetilde R_{ij}-R_{ij}|
        \leq
        \|e_i^*E\|_2\|A[I,j]\|_2
        \leq
        \eta\|A[I,j]\|_2.
\]
This proves the proposition.
\end{proof}

\begin{remark}[Numerical swap decisions]
Suppose that a row swap is accepted because the computed coefficient satisfies
\[
        |\widetilde C_{pq}|>1+\varepsilon.
\]
If
\[
        |\widetilde C_{pq}-C_{pq}|\leq \delta,
\]
then
\[
        |C_{pq}|
        \geq
        |\widetilde C_{pq}|-\delta
        >
        1+\varepsilon-\delta.
\]
Therefore, if \(\varepsilon>\delta\), the exact coefficient still satisfies
\[
        |C_{pq}|>1,
\]
and the exact volume increases. The same argument applies to column swaps through
\(\widetilde R\).
\end{remark}

This observation does not replace the exact monotonicity theorem. It only explains why the exact
theory remains relevant when the inverse is computed with sufficient accuracy. In practice, this
also motivates monitoring the inverse residual or recomputing the inverse when rank-one updates
become unstable.

\color{black}

\section{Numerical experiments}
\label{sec:experiments}

We now illustrate the behavior of the proposed methods.
The purpose is not to compare all possible quaternion low-rank approximation algorithms.
Instead, we focus on three simple questions.
First, does quaternion Greedy MaxVol give useful CUR approximations for RGB images?
Second, can RectMaxVol be used as a preconditioner for ill-conditioned quaternion least-squares systems?
Third, can RectMaxVol select representative rows in real quaternion-valued motion data?

All experiments were run with the number of threads limited to $8$. The image experiments use RGB images encoded as pure quaternion matrices,
\[
    Q = 0 + R i + G j + B k,
\]
where $R$, $G$ and $B$ are the red, green and blue channels. The reconstruction is converted back to RGB by extracting the imaginary components.

The code used to generate the numerical experiments is available at
\href{https://github.com/Kadaverciant/Quaternion-MaxVol-Submatrix-Selection.git}{our GitHub repository}.
All computations are performed with native quaternion arithmetic. In particular,
the basic quaternion linear algebra routines used by the proposed algorithms,
including the operations on quaternion bases and the auxiliary matrix routines,
are provided by the QuatIca package~\cite{leplat2026quaticaadvancednumericallinear}.


\subsection{RGB image approximation}

We first test quaternion CUR approximation on RGB images from the JPEG Happywhale dataset \cite{Dataset}.
The goal is to check whether selected rows and columns can reconstruct the image with increasing accuracy as the core size grows.

For the Greedy MaxVol experiment, we use QR initialization and compare two inverse strategies: the complex embedding inverse and the Newton--Schulz quaternion pseudoinverse.
The selected indices were the same in the tested cases, so the reconstruction errors coincide.
The main difference is the running time.

\begin{table}[ht!]
\centering
\caption{Greedy MaxVol with QR initialization on RGB images. The table reports mean values over the tested images. The Newton--Schulz variant is quaternion-native, while the $\Phi$ variant computes the inverse through the complex embedding.}
\label{tab:rgb_greedy_summary}
\begin{tabular}{cccccc}
\toprule
$k$ & inverse strategy & time (s) & iterations & relF & PSNR \\
\midrule
5  & $\Phi$ + rank-one updates & 0.027 & 7.5  & 0.296 & 17.004 \\
5  & NS + rank-one updates     & 0.036 & 7.5  & 0.296 & 17.004 \\
10 & $\Phi$ + rank-one updates & 0.075 & 14.2 & 0.254 & 18.527 \\
10 & NS + rank-one updates     & 0.104 & 14.2 & 0.254 & 18.527 \\
25 & $\Phi$ + rank-one updates & 0.132 & 28.0 & 0.204 & 20.921 \\
25 & NS + rank-one updates     & 0.177 & 28.0 & 0.204 & 20.921 \\
50 & $\Phi$ + rank-one updates & 0.344 & 37.8 & 0.169 & 23.127 \\
50 & NS + rank-one updates     & 0.550 & 37.8 & 0.169 & 23.127 \\
\bottomrule
\end{tabular}
\end{table}

The expected behavior is observed. Increasing $k$ improves the reconstruction quality.
The relative Frobenius error decreases from $0.296$ at $k=5$ to $0.169$ at $k=50$, while the PSNR increases from $17.004$ dB to $23.127$ dB.
The $\Phi$-based inverse is faster, while the Newton--Schulz version is closer to a quaternion-native implementation.

\begin{figure}[ht!]
\centering
\maybeincludegraphics[width=1\linewidth]{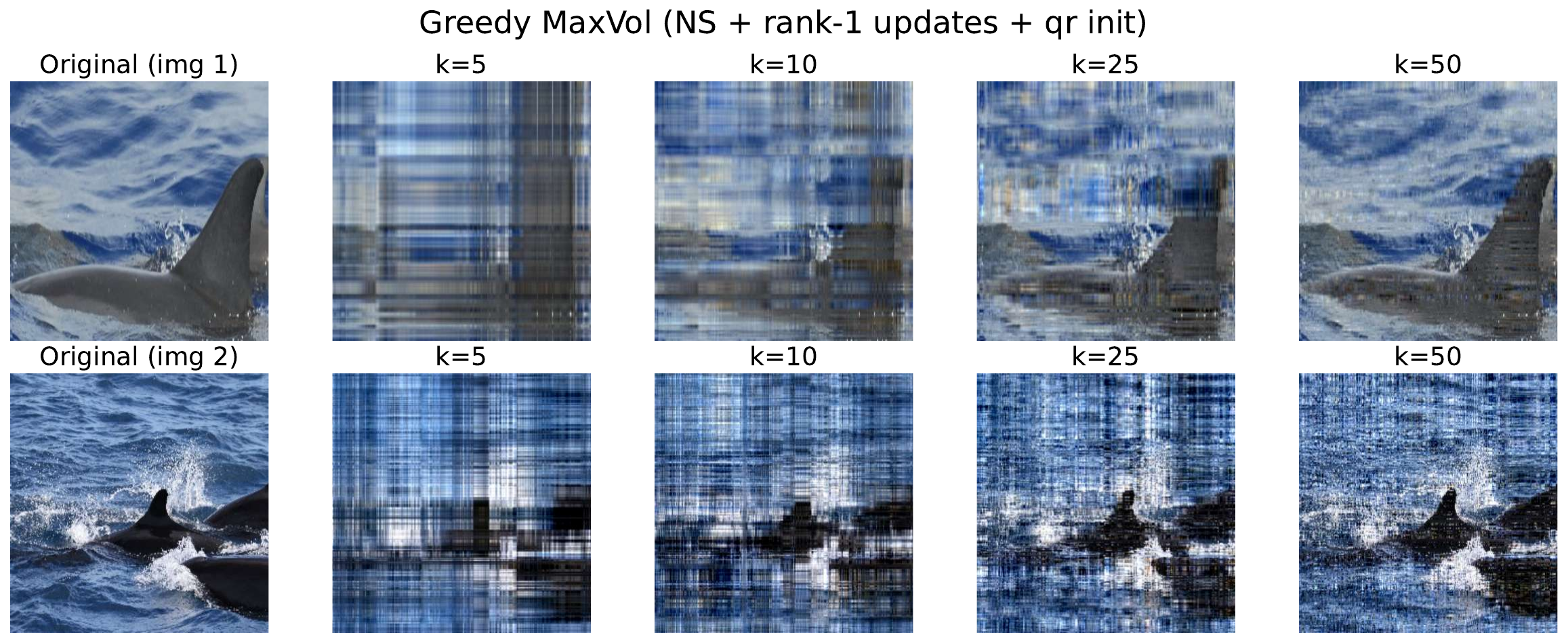}
\caption{Example of RGB reconstruction with quaternion Greedy MaxVol using Newton--Schulz pseudoinverse updates and QR initialization.}
\label{fig:rgb_greedy_demo}
\end{figure}

\begin{figure}[ht!]
\centering
\maybeincludegraphics[width=1\linewidth]{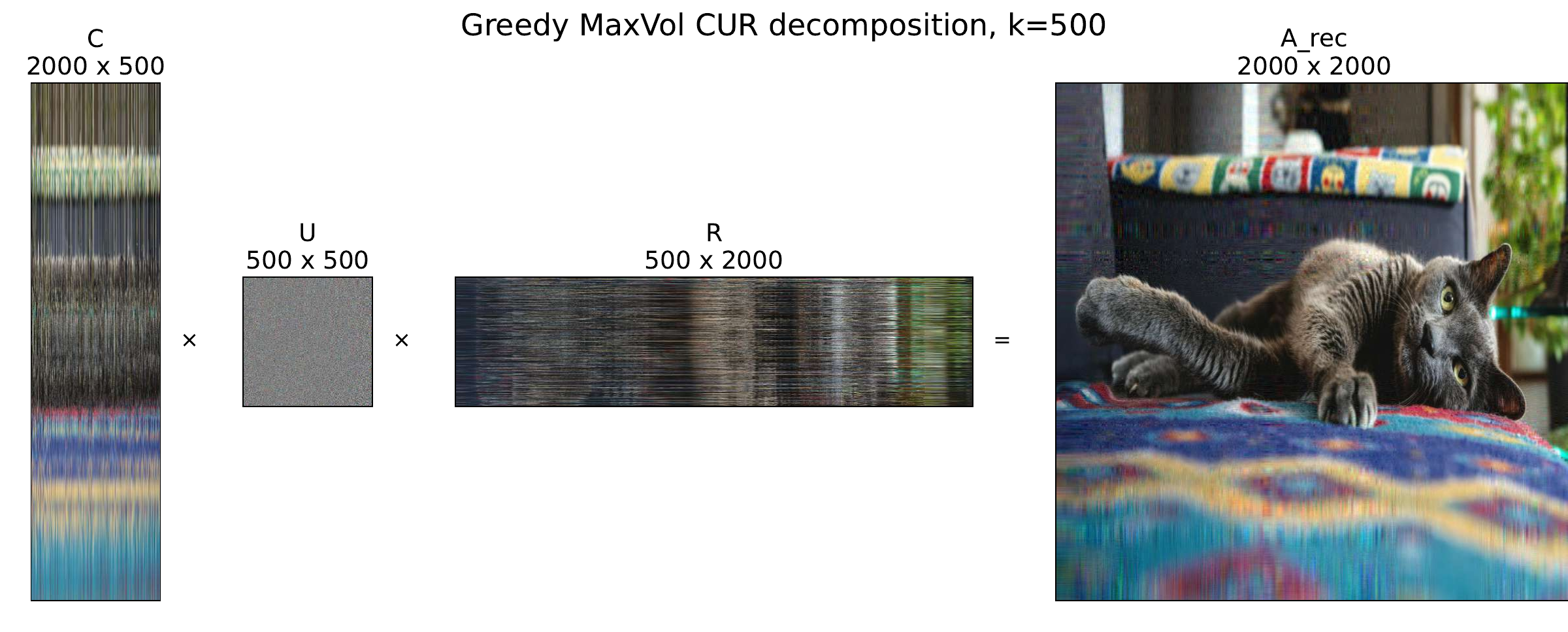}
\caption{Example of quaternion CUR decomposition for high resolution $2000 \times 2000$ RGB image with core size $k=500$.}
\label{fig:rgb_cur_cat_k500}
\end{figure}

\begin{figure}[hbt]
\centering
\maybeincludegraphics[width=1\linewidth]{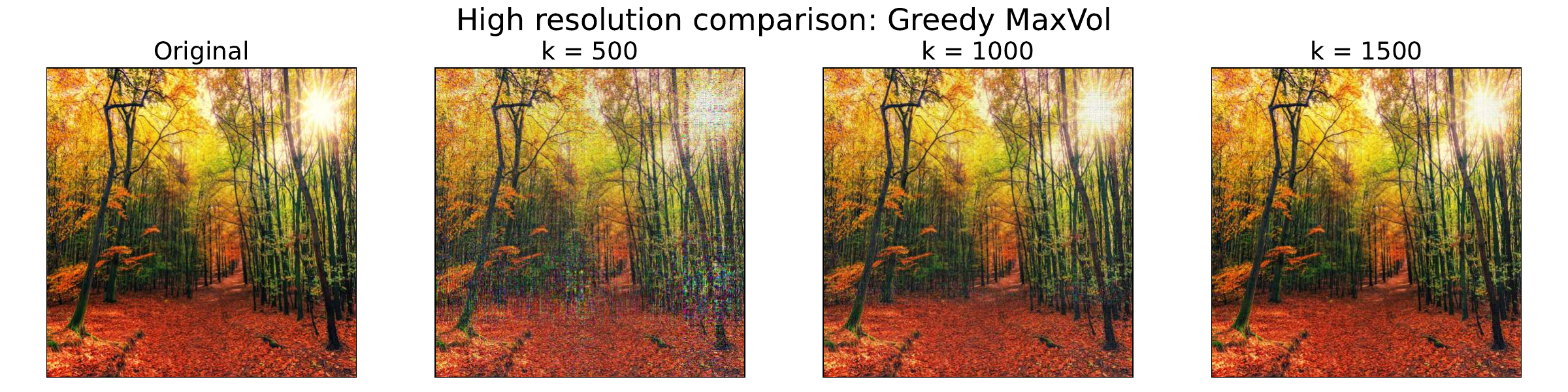}
\caption{Reconstruction demonstration for high resolution $2000 \times 2000$ image and core sizes $k \in [500, 1000, 1500]$.}
\label{fig:final_comparison_greedy_maxvol_demo}
\end{figure}

We also tested a high-resolution image resized to $2000\times 2000$ \cite{HighResImage}.
The reconstruction quality for different values of $k$ is shown in Fig.\ \ref{fig:final_comparison_greedy_maxvol_demo}.
Table~\ref{tab:large_rgb_greedy} reports the results for Greedy MaxVol only. This experiment mainly checks scalability of the basic method.

\begin{table}[ht!]
\centering
\caption{Greedy MaxVol on a $2000\times 2000$ RGB image.}
\label{tab:large_rgb_greedy}
\begin{tabular}{cccccc}
\toprule
$k$ & time (s) & relF & rel2 & MSE RGB & PSNR \\
\midrule
500  & 235.136  & 0.853 & 0.394 & 0.067 & 11.751 \\
1000 & 550.998  & 0.578 & 0.293 & 0.029 & 15.363 \\
1500 & 1177.220 & 0.252 & 0.134 & 0.006 & 22.143 \\
\bottomrule
\end{tabular}
\end{table}

The quality improves significantly with $k$. This is especially visible between $k=1000$ and $k=1500$.
The cost is also substantial, which is expected because the current implementation is not yet optimized for very large images.

\FloatBarrier
\subsubsection{Tensor-CUR comparison}

\begin{figure}[ht!]
\centering
\maybeincludegraphics[width=1\linewidth]{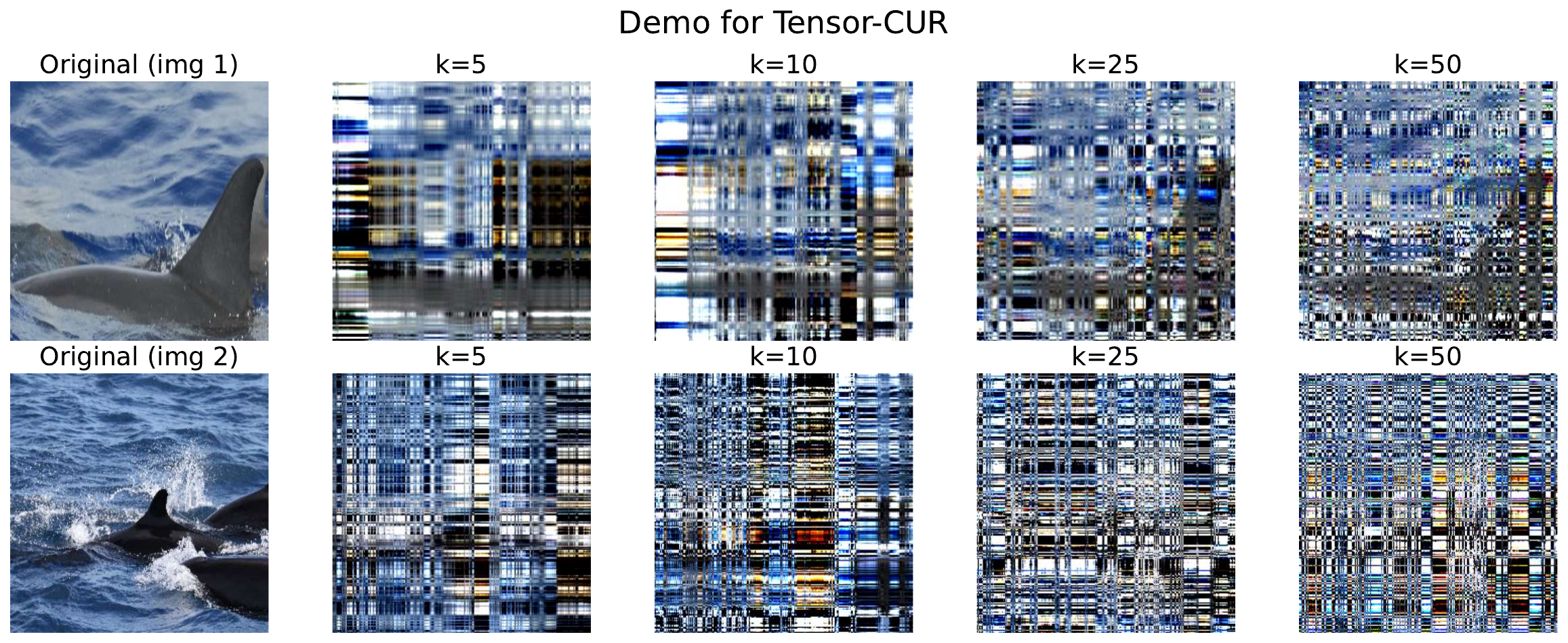}
\caption{Example of RGB reconstruction with Tensor-CUR.}
\label{fig:demo_tensor_cur}
\end{figure}

\begin{figure}[ht!]
\centering
\maybeincludegraphics[width=1\linewidth]{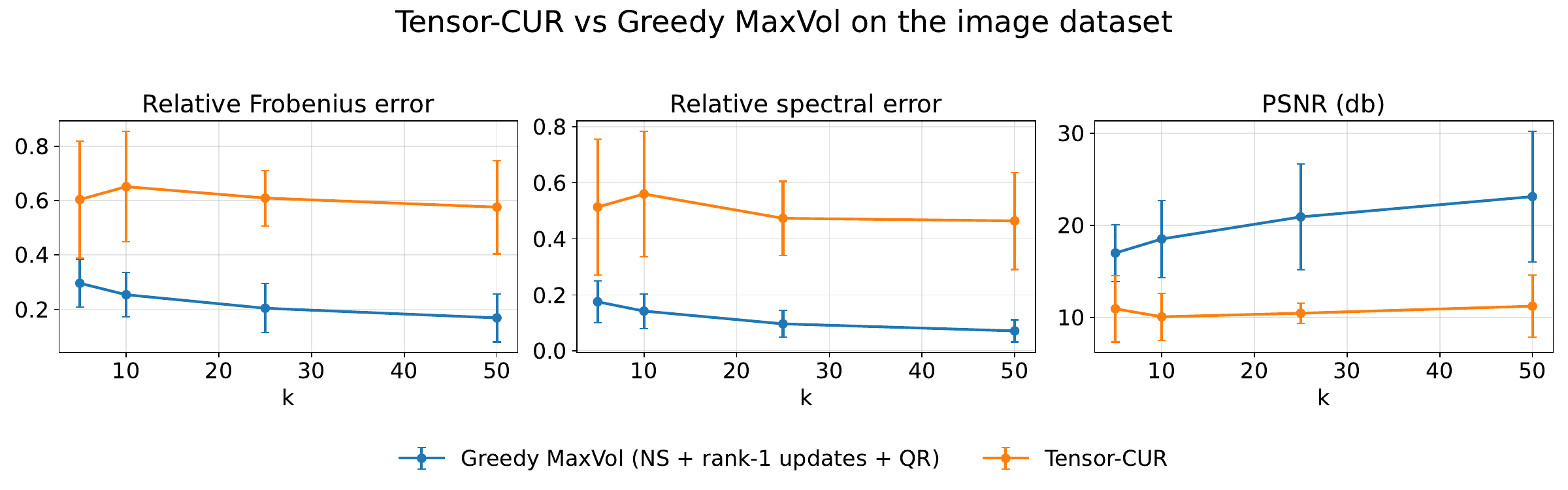}
\caption{Comparison between Greedy MaxVol and Tensor-CUR performance on Happywhale dataset.}
\label{fig:tensor_cur_vs_greedy_maxvol_metrics}
\end{figure}

\begin{figure}[ht!]
\centering
\maybeincludegraphics[width=1\linewidth]{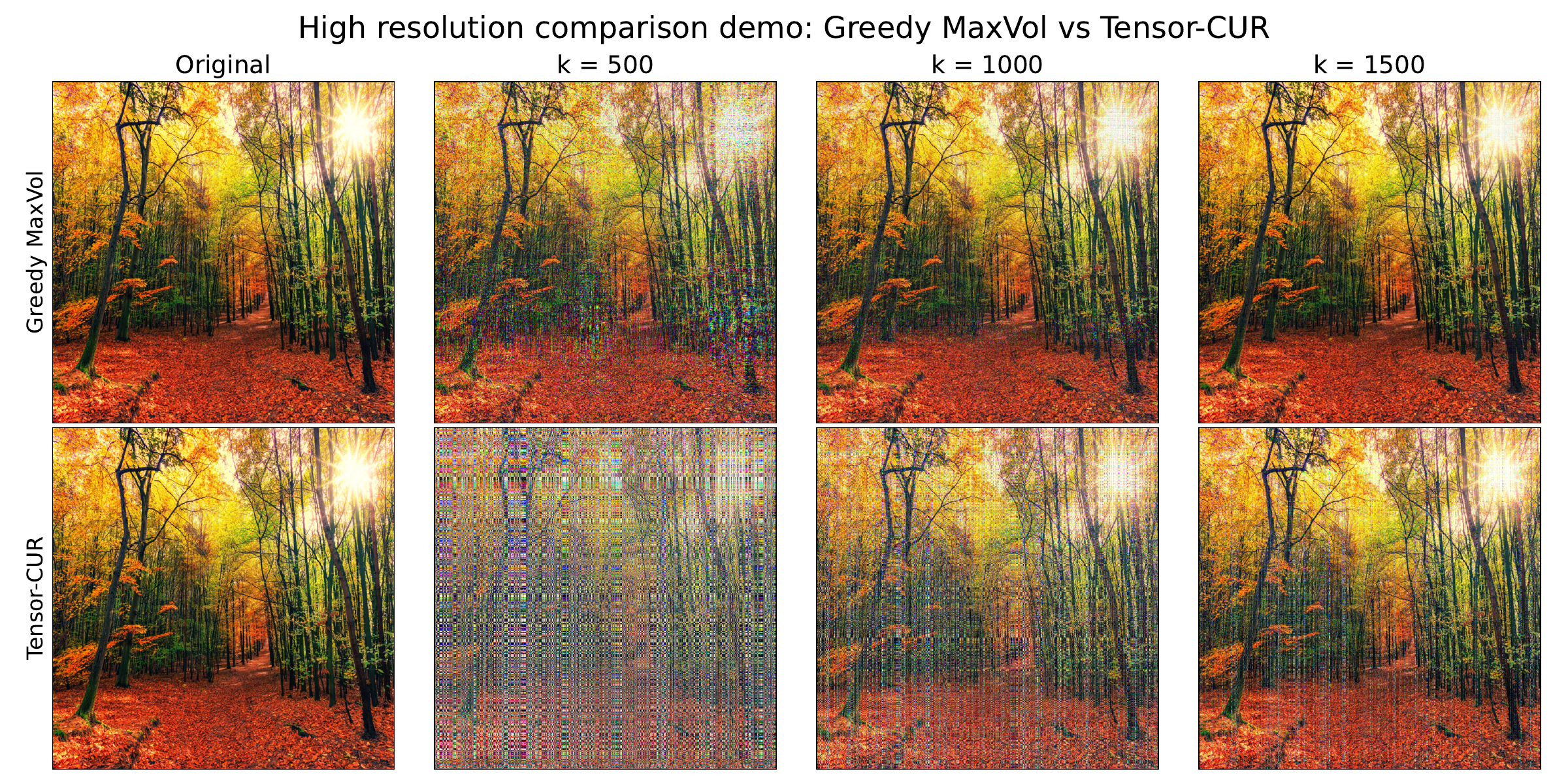}
\caption{Reconstruction demonstration of Greedy MaxVol and Tensor-CUR for high resolution $2000 \times 2000$ image and core sizes $k \in [500, 1000, 1500]$.}
\label{fig:final_comparison_greedy_tensor_demo}
\end{figure}

\begin{figure}[ht!]
\centering
\maybeincludegraphics[width=1\linewidth]{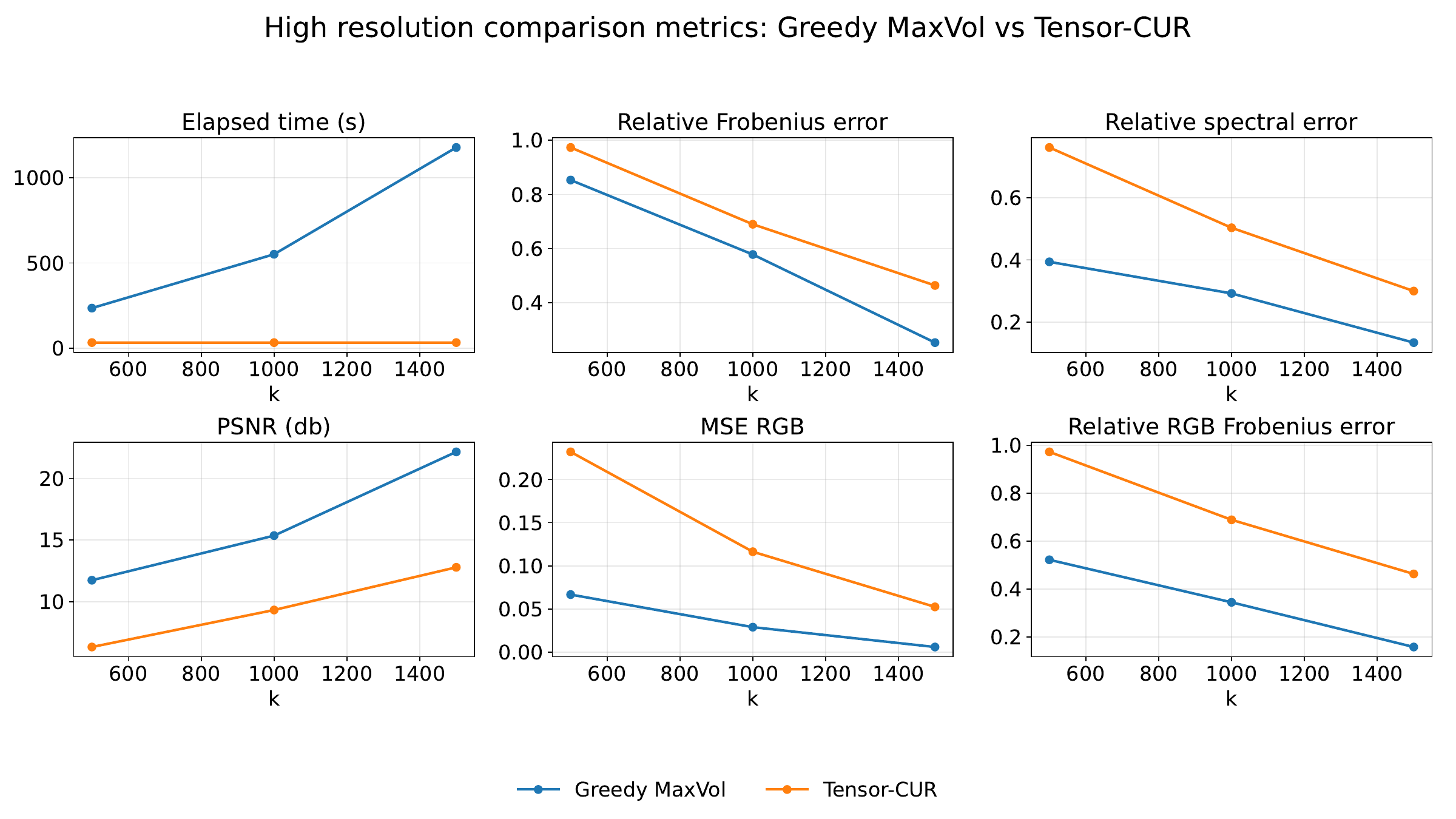}
\caption{Comparison between Greedy MaxVol and Tensor-CUR for high resolution $2000 \times 2000$ image and core sizes $k \in [500, 1000, 1500]$.}
\label{fig:final_comparison_greedy_tensor_metrics}
\end{figure}

As an additional baseline, we also tested the Tensor-CUR sampling idea of Mahoney, Maggioni, and Drineas\ \cite{mahoney2006tensor} on the same RGB data,
using norm-proportional random sampling without replacement, and $c=r=k$.
Greedy MaxVol outperforms Tensor-CUR across all metrics except execution time.
Thus we can state that Tensor-CUR is useful as a very cheap tensor baseline, but the maximum-volume selection
gives substantially better reconstructions in these tests.

\FloatBarrier
\subsection{RectMaxVol-based preconditioning for quaternion least squares}

We next test RectMaxVol as a row-selection tool for preconditioning. Let
\[
    A\in\Hq^{m\times n}, \qquad m\gg n,
\]
be full column rank. We consider
\begin{equation}
\label{eq:quat_ls}
    \min_{x\in\Hq^n}\|Ax-b\|_2.
\end{equation}
For a right preconditioner $T$, we solve
\[
    \min_{y\in\Hq^n}\|ATy-b\|_2,
    \qquad x=Ty.
\]
Square MaxVol selects $n$ rows and uses $T=B^{-1}$. RectMaxVol selects $r\ge n$ rows, computes $S=A[I,:]=QR$, and uses $T=R^{-1}$. In the implementation used here, the iterative least-squares solve is applied to the real representation of the quaternion system, while the row selection and the preconditioners are built from quaternion operations.

We generated synthetic quaternion matrices with prescribed singular values. The
matrices were of sizes \(150\times 10\), \(300\times 20\), and \(600\times 40\).
The target condition number in the reported runs is \(10^6\), with singular
values distributed logarithmically between \(1\) and \(10^{-6}\). For each size
and noise level, we performed \(100\) independent trials. The right-hand side is
\[
    b=Ax_{\rm true}+\eta,
\]
where \(\eta\) is a quaternion Gaussian noise vector. Table~\ref{tab:preconditioning}
reports median values for noise level \(\alpha=10^{-8}\).

\begin{table}[ht!]
\centering
\caption{Synthetic quaternion least-squares preconditioning results for noise level $\alpha=10^{-8}$. Median values over $100$ trials are reported.}
\label{tab:preconditioning}
\resizebox{\linewidth}{!}{%
\begin{tabular}{llccccc}
\toprule
Size & Method & $\kappa_2(A_{\rm prec})$ & Iter. & Rel. residual & Rel. sol. error & Total time (s) \\
\midrule
small & No preconditioner    & $1.000\cdot 10^6$ & 27.0  & $1.947\cdot 10^{-8}$ & $5.042\cdot 10^{-4}$ & 0.005 \\
small & Random rows + QR     & $4.108\cdot 10^0$ & 10.0  & $1.946\cdot 10^{-8}$ & $5.260\cdot 10^{-4}$ & 0.002 \\
small & Norm rows + QR       & $4.803\cdot 10^0$ & 10.0  & $1.951\cdot 10^{-8}$ & $5.214\cdot 10^{-4}$ & 0.002 \\
small & Square MaxVol        & $4.972\cdot 10^0$ & 10.0  & $1.951\cdot 10^{-8}$ & $5.365\cdot 10^{-4}$ & 0.006 \\
small & RectMaxVol + QR      & $2.616\cdot 10^0$ & 10.0  & $1.944\cdot 10^{-8}$ & $5.042\cdot 10^{-4}$ & 0.028 \\
\midrule
medium & No preconditioner   & $1.000\cdot 10^6$ & 115.0 & $2.944\cdot 10^{-7}$ & $3.306\cdot 10^{-2}$ & 0.018 \\
medium & Random rows + QR    & $4.797\cdot 10^0$ & 19.0  & $8.999\cdot 10^{-8}$ & $4.101\cdot 10^{-3}$ & 0.009 \\
medium & Norm rows + QR      & $5.088\cdot 10^0$ & 19.0  & $9.388\cdot 10^{-8}$ & $5.397\cdot 10^{-3}$ & 0.010 \\
medium & Square MaxVol       & $8.401\cdot 10^0$ & 23.0  & $2.349\cdot 10^{-8}$ & $8.654\cdot 10^{-4}$ & 0.018 \\
medium & RectMaxVol + QR     & $3.190\cdot 10^0$ & 17.0  & $7.200\cdot 10^{-8}$ & $3.342\cdot 10^{-3}$ & 0.068 \\
\midrule
large & No preconditioner    & $1.000\cdot 10^6$ & 500.0 & $1.263\cdot 10^{-6}$ & $1.861\cdot 10^{-1}$ & 0.140 \\
large & Random rows + QR     & $5.049\cdot 10^0$ & 26.0  & $1.523\cdot 10^{-7}$ & $6.688\cdot 10^{-3}$ & 0.032 \\
large & Norm rows + QR       & $5.312\cdot 10^0$ & 26.0  & $1.662\cdot 10^{-7}$ & $7.429\cdot 10^{-3}$ & 0.034 \\
large & Square MaxVol        & $1.452\cdot 10^1$ & 38.0  & $3.904\cdot 10^{-7}$ & $1.844\cdot 10^{-2}$ & 0.070 \\
large & RectMaxVol + QR      & $3.638\cdot 10^0$ & 23.0  & $9.168\cdot 10^{-8}$ & $4.358\cdot 10^{-3}$ & 0.282 \\
\bottomrule
\end{tabular}}
\end{table}

The effect of preconditioning is clear. Without preconditioning, the large problem reaches the iteration limit of $500$. With RectMaxVol, the median condition number is reduced to $3.638$, and the median number of iterations decreases to $23$. RectMaxVol is more expensive to set up than the simple baselines, but it gives the best conditioning in these tests.

\begin{figure}[ht!]
\centering
\maybeincludegraphics[width=1\linewidth]{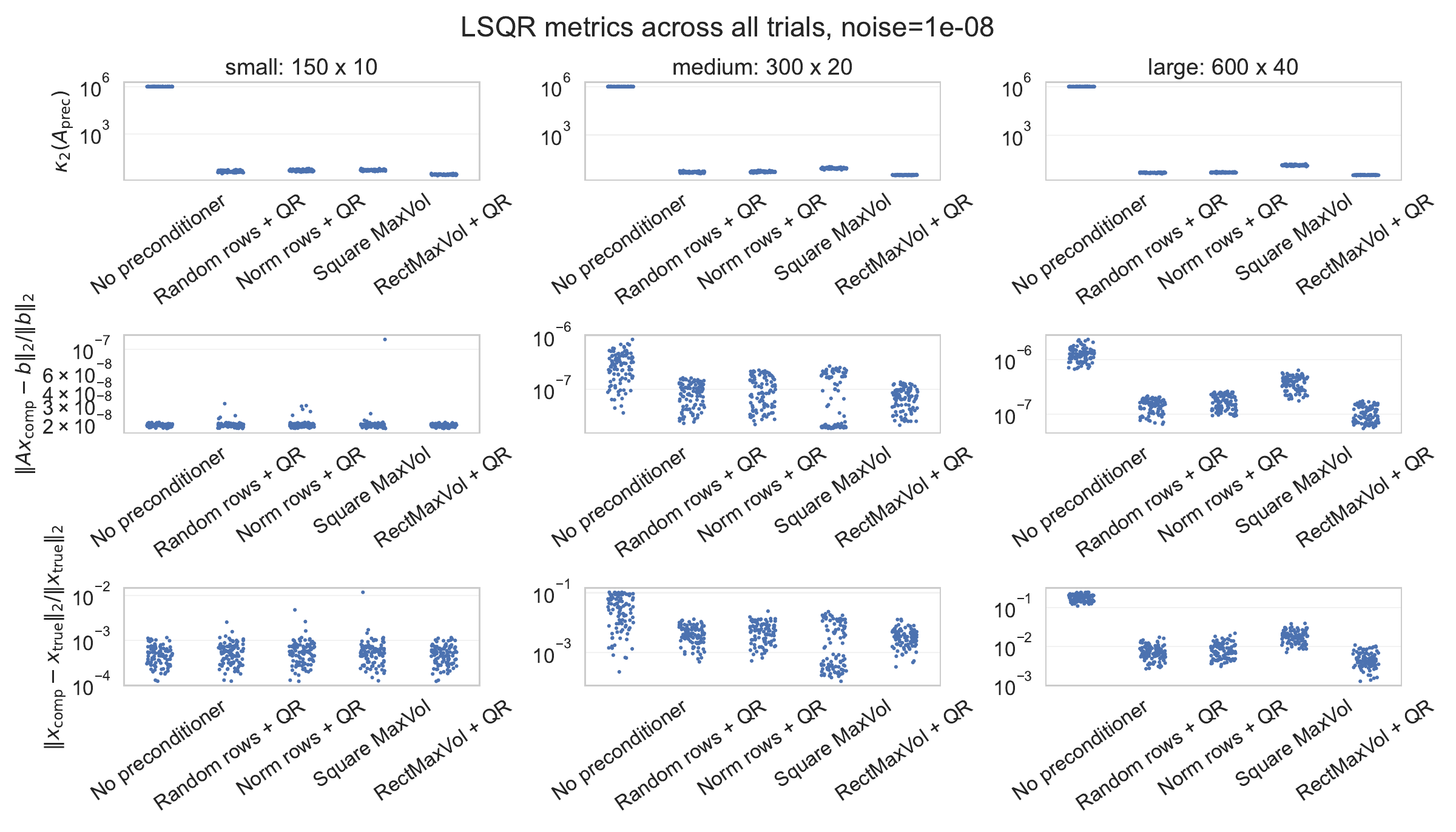}
\caption{Synthetic least-squares metrics for noise level $\alpha=10^{-8}$.}
\label{fig:ls_noise_1e8}
\end{figure}

\subsection{RectMaxVol row selection on motion-capture data}

We now test RectMaxVol on real quaternion-valued data. The data source is the CMU Graphics Lab Motion Capture Database \cite{CMUMocap,re3dataCMUMocap}. We use the BVH conversion mirror \cite{CMUMocapBVH}. For each BVH file, we parse the skeleton and keep $17$ major joints. The three-dimensional position of a joint is encoded as a pure quaternion,
\[
    (x,y,z) \mapsto 0+xi+yj+zk.
\]
For a motion clip with $N$ frames, this gives a quaternion pose matrix $P\in\Hq^{N\times 17}$. Since we want to select important changes in motion, we use frame-to-frame differences,
\[
    A_t=P_{t+1,:}-P_{t,:}, \qquad t=0,\ldots,N-2.
\]
Thus each row of $A\in\Hq^{(N-1)\times 17}$ corresponds to one motion change.

After a method selects row indices $I$, we evaluate how well all motion changes can be reconstructed from the selected rows. The selection itself is quaternion-based for RectMaxVol, while the reconstruction error is evaluated after converting the quaternion change matrix to its complex $\Phi$-representation
\[
    X=\Phi(A)\in\C^{2(N-1)\times 34}.
\]
Each selected quaternion row corresponds to two rows in the complex embedding. If $m=N-1$, we use the embedded row set
\[
    I_{\Phi}=I\cup\{i+m:\; i\in I\}.
\]
With $C_{\Phi}=X[I_{\Phi},:]$, we compute
\[
    W^\star=\argmin_W \|W C_{\Phi}-X\|_F,
    \qquad \widehat X=W^\star C_{\Phi},
\]
using a least-squares solve. The reported error is
\[
    \frac{\|X-\widehat X\|_F}{\|X\|_F}.
\]


We compare RectMaxVol with random rows, norm rows, KMeans centers~\cite{macqueen1967kmeans}, and leverage scores~\cite{drineas2012leverage}. All methods select the same number of rows as RectMaxVol. The random baseline uses one random draw for each processed file. KMeans and leverage scores are computed on the complex embedded feature matrix used for the final reconstruction metric; KMeans is run with $10$ initializations and each selected row is the nearest actual row to a cluster center.

\begin{table}[ht!]
\centering
\caption{MoCap change-frame reconstruction for one representative BVH file.}
\label{tab:mocap_single}
\begin{tabular}{lccc}
\toprule
Method & selected rows & relative Frobenius error & setup time (s) \\
\midrule
RectMaxVol    & 23 & 0.099553 & 0.199998 \\
Random rows   & 23 & 0.165409 & 0.000138 \\
Norm rows     & 23 & 0.124786 & 0.000487 \\
KMeans centers& 23 & 0.118046 & 1.666353 \\
Leverage scores & 23 & 0.123121 & 0.012411 \\
\bottomrule
\end{tabular}
\end{table}

\begin{figure}[t]
\centering
\maybeincludegraphics[width=1\linewidth]{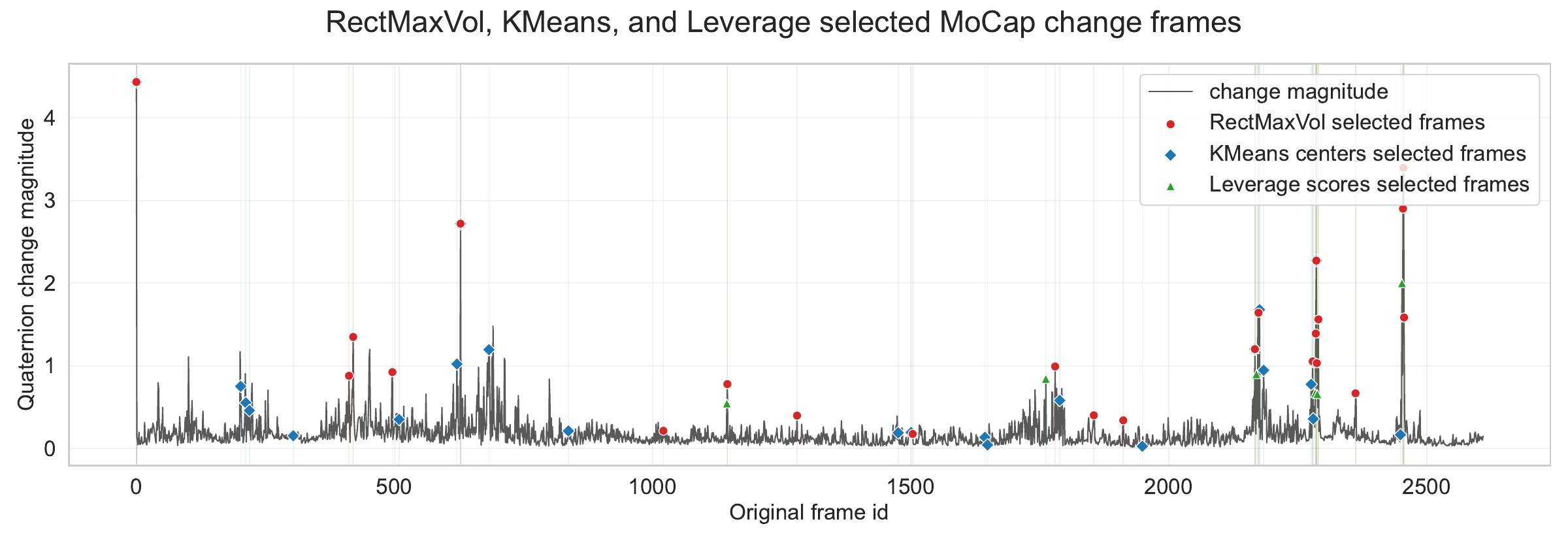}
\caption{Selected motion-change frames for RectMaxVol and the baseline methods on one representative BVH file.}
\label{fig:mocap_selected_frames}
\end{figure}

We also processed $1000$ BVH records. Table~\ref{tab:mocap_all} reports median results. RectMaxVol gives the smallest median reconstruction error, with a moderate advantage over KMeans centers. It is also faster than KMeans in setup time, but slower than norm rows and leverage scores.

\begin{table}[ht!]
\centering
\caption{Median MoCap reconstruction results over $1000$ processed BVH records.}
\label{tab:mocap_all}
\begin{tabular}{lccc}
\toprule
Method & median rows & median error & median setup time (s) \\
\midrule
RectMaxVol      & 21.0 & 0.048421 & 0.182165 \\
Random rows     & 21.0 & 0.153034 & 0.001185 \\
Norm rows       & 21.0 & 0.065842 & 0.001961 \\
KMeans centers  & 21.0 & 0.049633 & 0.357306 \\
Leverage scores & 21.0 & 0.059677 & 0.009400 \\
\bottomrule
\end{tabular}
\end{table}

\begin{figure}[ht!]
\centering
\maybeincludegraphics[width=1\linewidth]{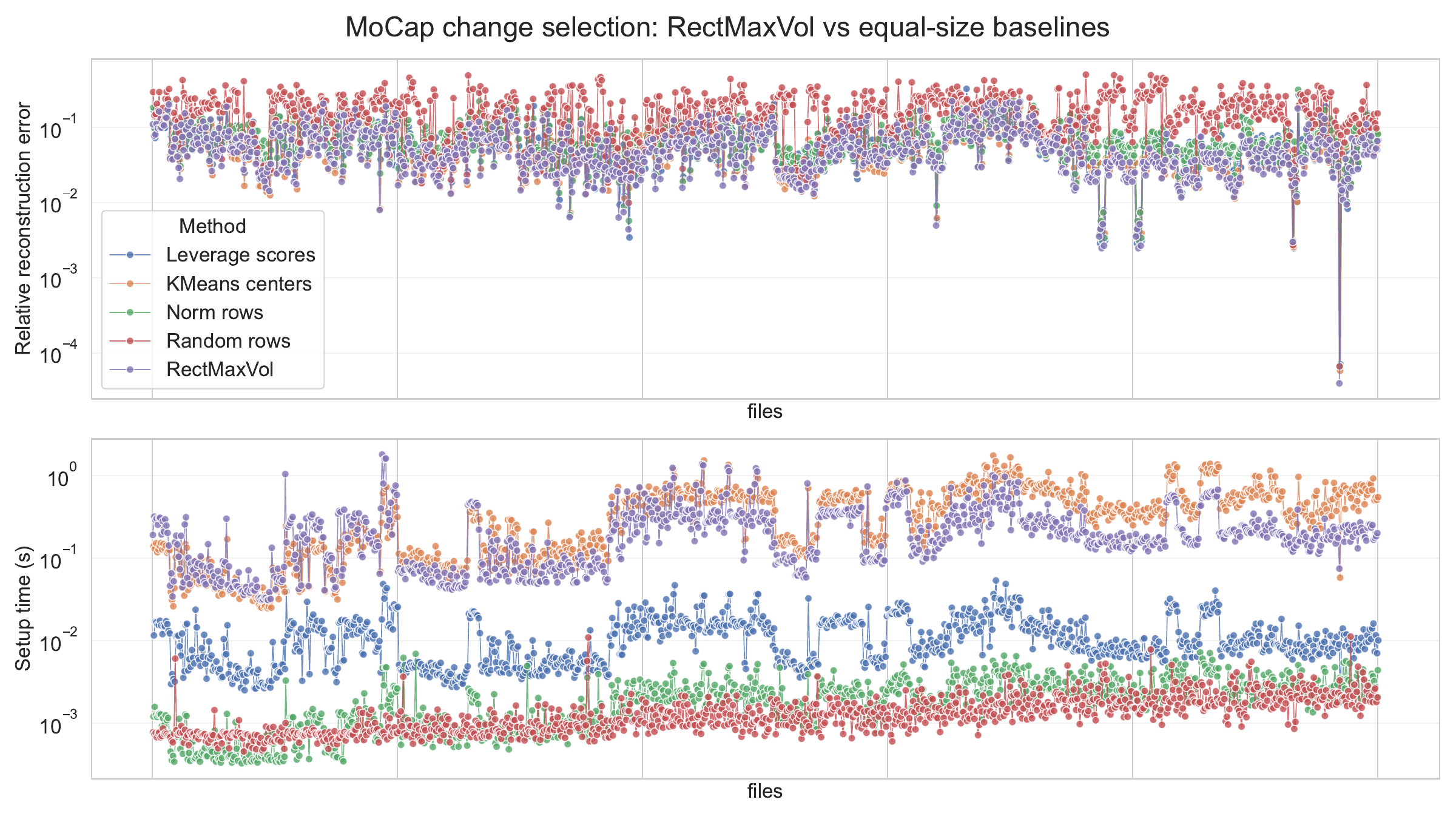}
\caption{MoCap reconstruction error and setup time over $1000$ processed BVH records.}
\label{fig:mocap_metrics}
\end{figure}

\begin{figure}[ht!]
\centering
\maybeincludegraphics[width=1\linewidth]{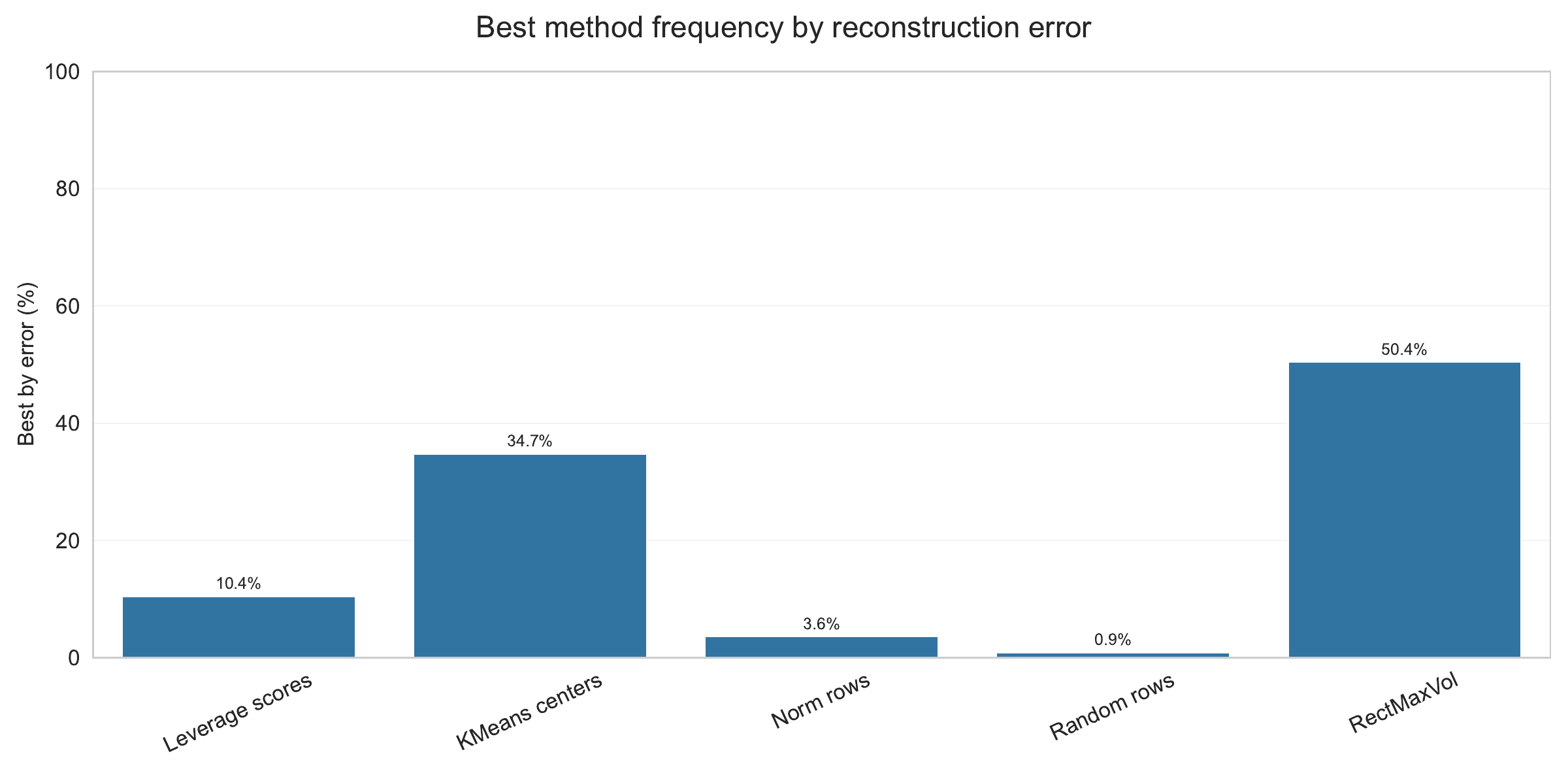}
\caption{Percentage of BVH records where each method gives the best reconstruction error. RectMaxVol is the best method for $50.4\%$ of the processed records.}
\label{fig:mocap_best_percent}
\end{figure}

Fig.~\ref{fig:mocap_best_percent} summarizes how often each method gives the best reconstruction error.
RectMaxVol is the best method for $50.4\%$ of the processed records.

\begin{figure}[ht!]
\centering
\maybeincludegraphics[width=1\linewidth]{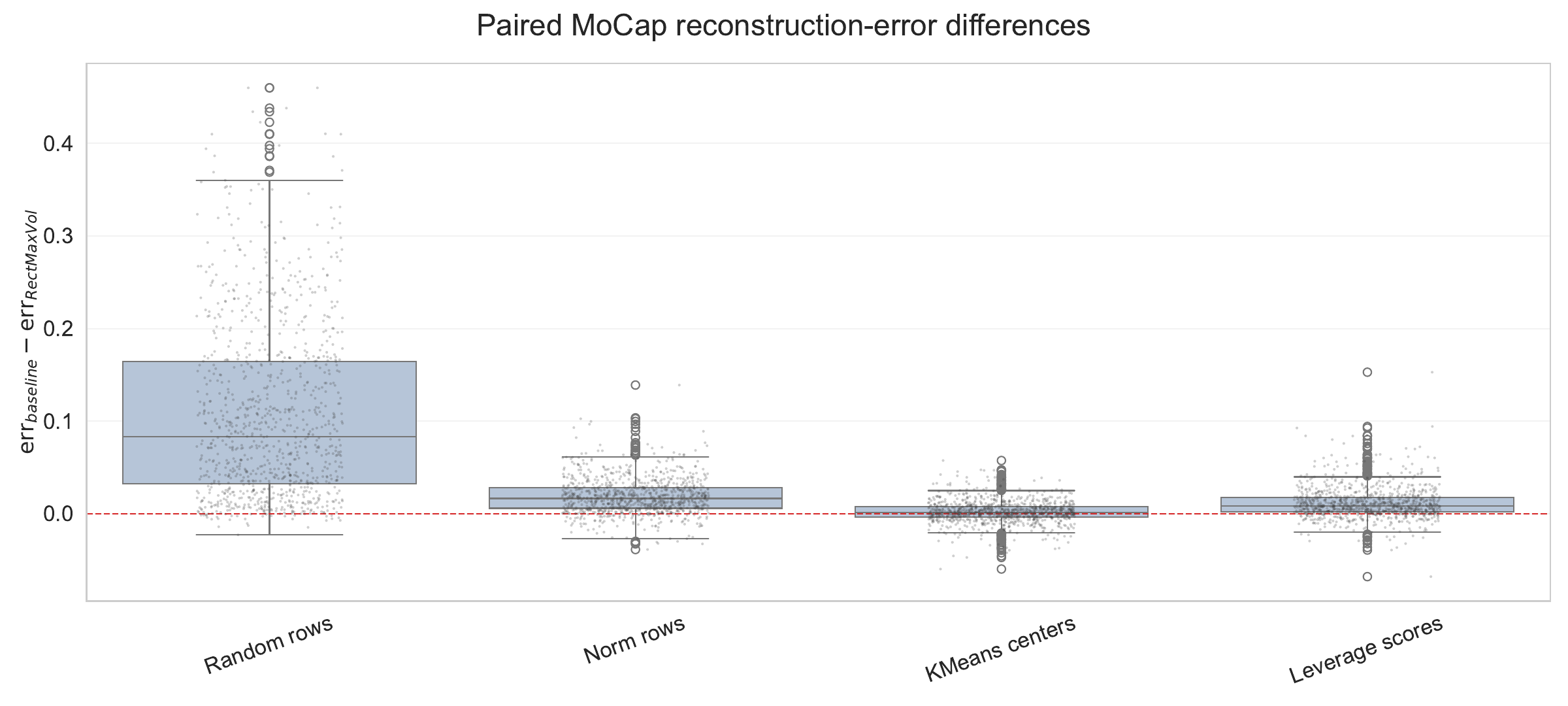}
\caption{Paired differences between each baseline reconstruction error and the RectMaxVol reconstruction error over $1000$ processed BVH records. Positive values mean that RectMaxVol gives a smaller error for the same BVH file.}
\label{fig:mocap_paired_differences}
\end{figure}

Fig.~\ref{fig:mocap_paired_differences} provides a paired comparison for the same records by plotting $\mathrm{err}_{\mathrm{baseline}}-\mathrm{err}_{\mathrm{RectMaxVol}}$ for each baseline.
The median paired difference is positive for all baselines: $0.001433$ for KMeans centers, $0.008638$ for leverage scores, $0.016495$ for norm rows, and $0.083408$ for random rows.
Thus, RectMaxVol improves over KMeans only moderately in the median paired error, while the gaps to leverage scores, norm rows, and random rows are larger.

\begin{figure}[ht!]
\centering
\maybeincludegraphics[width=1\linewidth]{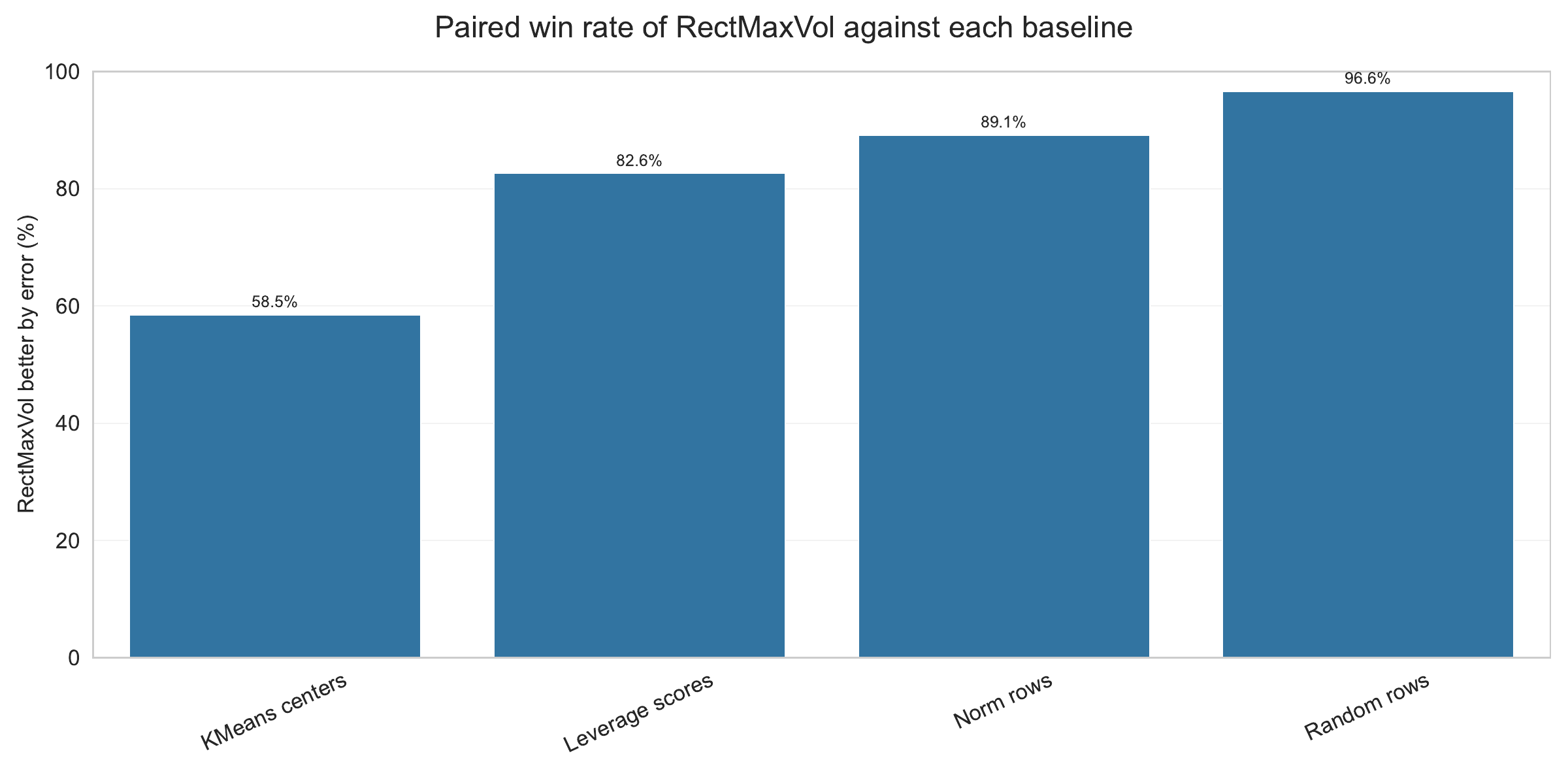}
\caption{Percentage of BVH records where RectMaxVol gives a smaller reconstruction error than each baseline method in the paired comparison.}
\label{fig:mocap_paired_win_rate}
\end{figure}

Fig.~\ref{fig:mocap_paired_win_rate} shows the corresponding paired win rates:
RectMaxVol gives a smaller reconstruction error than KMeans centers for $58.5\%$ of the records, than leverage scores for $82.6\%$, than norm rows for $89.1\%$, and than random rows for $96.6\%$.

Overall, these experiments show that RectMaxVol can be used as a practical row-selection method for quaternion motion data.
It selects a small number of informative motion-change frames and gives the best median reconstruction quality among the tested methods.
The main cost of this improvement is the setup time, because RectMaxVol performs an iterative row-selection procedure.

\newpage
\section{Conclusion and perspectives}
\label{sec:conclusion}

We studied maximum-volume submatrix selection for quaternion matrices. The main point was to adapt the classical MaxVol and RectMaxVol ideas while keeping track of quaternion noncommutativity. We defined the quaternion volume using singular values and the Study determinant, derived the rank-one update formulas needed for row and column replacements, and proved that successful Greedy MaxVol swaps increase the quaternion volume when exact inverses are used. We also connected the stopping criterion with a quasi-dominance property, proved an exact quaternion CUR identity in the full-rank case, and derived an interpolation stability bound for the associated CUR approximation. For the rectangular case, we derived an append-row pseudoinverse update and showed how RectMaxVol leads to a natural right preconditioner for overdetermined quaternion least-squares problems.

The numerical experiments support this theoretical picture. Greedy MaxVol gives stable quaternion CUR approximations of RGB images, and the quality improves as the selected core grows. RectMaxVol gives strong preconditioners for ill-conditioned quaternion least-squares systems, reducing both the condition number and the number of iterations. It also provides an effective row-selection method for quaternion motion-capture data, where it selects representative motion changes and compares favorably with standard baselines.

Several directions seem promising. A first one is to combine quaternion MaxVol selection with
randomized quaternion low-rank approximation. In this setting, sketching could first build a
small quaternion basis, while MaxVol would then extract actual rows and columns, leading to
fast and interpretable CUR approximations. A second direction is to use MaxVol ideas inside
structured constrained models. For example, in quaternion analogues of separable or coseparable
models, the core is itself a submatrix of the data. This may lead to classes of selection problems
that are easier than the general maximum-volume problem, and where exact or stable recovery
can be proved under geometric assumptions. Finally, it would be interesting to extend the present
framework beyond classical quaternions, for instance to matrices over split quaternions~\cite{ablamowicz2020split} and Clifford algebras~\cite{cao2022clifford}, which naturally arise in signal processing and geometric computing.

\printbibliography

\appendix

\section{Quaternion Woodbury and rank-one update formulas}
\label{app:woodbury-updates}

We include the main inverse-update identities used in the algorithms of Section~\ref{sec:algorithms}.
They are standard in the real and complex cases, but in the quaternion setting the order of
multiplication has to be kept fixed. This appendix justifies the formulas stated in
Section~\ref{sec:algorithms}, in particular the row and column inverse updates used by
Algorithm~\ref{alg:q_greedy_maxvol}.

\begin{proposition}[Woodbury identity over \(\mathbb H\)]
Let \(A\in\mathbb H^{n\times n}\) be invertible, and let
\(U\in\mathbb H^{n\times k}\), \(C\in\mathbb H^{k\times k}\), and
\(V\in\mathbb H^{k\times n}\). Assume that \(C\) and
\[
    S=C^{-1}+VA^{-1}U
\]
are invertible. Then
\[
    (A+UCV)^{-1}
    =
    A^{-1}-A^{-1}US^{-1}VA^{-1}.
\]
\end{proposition}

\begin{proof}
Let
\[
    X=A^{-1}-A^{-1}US^{-1}VA^{-1}.
\]
We prove both inverse identities. First,
\[
\begin{aligned}
(A+UCV)X
&=I-US^{-1}VA^{-1}+UCVA^{-1}
  -UCVA^{-1}US^{-1}VA^{-1}       \\
&=I+U\left(C-S^{-1}-CVA^{-1}US^{-1}\right)VA^{-1}.
\end{aligned}
\]
Since
\[
    S=C^{-1}+VA^{-1}U,
\]
we have
\[
    CS=I+CVA^{-1}U,
\]
and therefore
\[
    CVA^{-1}US^{-1}=(CS-I)S^{-1}=C-S^{-1}.
\]
Thus \((A+UCV)X=I\).

Similarly,
\[
\begin{aligned}
X(A+UCV)
&=I-A^{-1}US^{-1}V+A^{-1}UCV
  -A^{-1}US^{-1}VA^{-1}UCV       \\
&=I+A^{-1}U\left(C-S^{-1}-S^{-1}VA^{-1}UC\right)V.
\end{aligned}
\]
Since
\[
    SC=I+VA^{-1}UC,
\]
multiplying on the left by \(S^{-1}\) gives
\[
    C=S^{-1}+S^{-1}VA^{-1}UC.
\]
Hence \(X(A+UCV)=I\). Thus \(X=(A+UCV)^{-1}\).
\end{proof}

Taking \(C=1\), \(U=u\), and \(V=v^*\), we obtain the quaternion
Sherman--Morrison formula
\[
    (A+uv^*)^{-1}
    =
    A^{-1}-A^{-1}u(1+v^*A^{-1}u)^{-1}v^*A^{-1}.
\]

If the \(q\)th column of \(A\) is replaced by a new column \(b\), then
\[
    A_{\rm new}=A+(b-a_q)e_q^*,
\]
where \(a_q\) is the old \(q\)th column. Hence
\[
    A_{\rm new}^{-1}
    =
    A^{-1}
    -
    A^{-1}(b-a_q)
    \left(1+e_q^*A^{-1}(b-a_q)\right)^{-1}
    e_q^*A^{-1}.
\]
If the \(q\)th row of \(A\) is replaced by a new row \(b^*\), then
\[
    A_{\rm new}=A+e_q(b^*-a_q^*),
\]
and
\[
    A_{\rm new}^{-1}
    =
    A^{-1}
    -
    A^{-1}e_q
    \left(1+(b^*-a_q^*)A^{-1}e_q\right)^{-1}
    (b^*-a_q^*)A^{-1}.
\]
These are the update formulas used in Algorithm~\ref{alg:q_greedy_maxvol}.

\section{Other quaternion MaxVol-type adaptations}
\label{app:other_algorithms}

This appendix briefly describes two additional algorithms that were adapted to the quaternion setting during the project. They are not used as main contributions in the paper because they did not improve over the simpler methods in the current experiments. Still, they are useful to document, since they may become relevant after further implementation work.

\subsection{Quaternion GMVA}

The alternating $h$-greedy maximal-volume algorithm tests several candidate swaps in each row or column phase. In the quaternion version, candidate quality is evaluated with the quaternion volume. The procedure is close to Greedy MaxVol, except that several candidate entries are considered before accepting a swap.

\begin{algorithm}[ht!]
\caption{Quaternion alternating $h$-greedy MaxVol}
\label{alg:q_gmva_appendix}
\begin{algorithmic}[1]
\Require $A\in\Hq^{m\times n}$, core size $k$, tolerance $\eps$, number of candidates $h$
\Ensure row indices $I$, column indices $J$, core $B=A[I,J]$
\State Initialize $I,J$ and set $B=A[I,J]$.
\For{each sweep}
    \State Compute a pseudoinverse or inverse of $B$.
    \State Form $R=A[:,J]B^\pinv$ and set $v_0=\vol(B)$.
    \For{$t=1,\ldots,h$}
        \State Choose the largest admissible entry $R_{pq}$ not already excluded.
        \State Form the trial core $B_{\rm try}$ obtained by replacing the $q$th row by row $p$.
        \If{$\vol(B_{\rm try})>v_0$}
            \State Accept the row swap and update $B$.
            \State \textbf{break}
        \EndIf
    \EndFor
    \State Repeat the same procedure for columns using $B^\pinv A[I,:]$.
\EndFor
\State \Return $I,J,B$
\end{algorithmic}
\end{algorithm}

In our tests, increasing $h$ did not provide a clear improvement. The selected cores were often the same as for $h=1$, while the method remained more expensive because it recomputed the core pseudoinverse more often.

\subsection{Adaptive quaternion Block DEIM}

Adaptive Block DEIM first computes a truncated QSVD and then selects row and column indices from the basis matrices. The quaternion version uses quaternion residualization and an internal one-sided quaternion MaxVol routine for block selection.

\begin{algorithm}[ht!]
\caption{Adaptive quaternion Block DEIM}
\label{alg:q_block_deim_appendix}
\begin{algorithmic}[1]
\Require $A\in\Hq^{m\times n}$, target rank $k$, block size $b$, parameter $\rho$, tolerance $\eps$
\Ensure row indices $I$, column indices $J$
\State Compute a truncated QSVD $A\approx U\Sigma V^*$ with $U\in\Hq^{m\times k}$ and $V\in\Hq^{n\times k}$.
\State Apply the adaptive selection routine to $U$ to obtain row indices $I$.
\State Apply the adaptive selection routine to $V$ to obtain column indices $J$.
\State Return $I,J$.
\end{algorithmic}
\end{algorithm}

The method is sensitive to the block size. In the current implementation it can also become expensive for large matrices because it relies on a QSVD and repeated internal MaxVol calls. For this reason, it is kept as a possible future direction rather than as a main method in the paper.

\end{document}